\documentclass[11pt]{amsart}

\usepackage[T1]{fontenc}
\usepackage[utf8]{inputenc}
\usepackage{graphicx} 
\usepackage{amsthm}
\usepackage{amssymb}
\usepackage{amsfonts}
\usepackage{xcolor}
\usepackage{tikz-cd}
\usepackage{mathtools}
\usepackage{enumerate}
\usepackage{comment}
\newtheorem{thm}{Theorem}[section]
\newtheorem{theorem}{Theorem}
\newtheorem{cor}[theorem]{Corollary}
\newtheorem{lm}[thm]{Lemma}
\newtheorem{pr}[thm]{Proposition}
\theoremstyle{definition}
\newtheorem{df}[thm]{Definition}
\theoremstyle{remark}
\newtheorem{cl}{Claim}
\newtheorem{rem}[thm]{Remark}

\newtheorem*{question}{Question}
\newcommand{\cftf}{$C(4)$--$T(4)$}
\newcommand{\cftfs}{$C(4)$--$T(4)$ }

\newcommand{\csvs}{$C(7)$ }
\newcommand{\cs}{$C(6)$}
\newcommand{\css}{$C(6)$ }

\usepackage[a4paper, left = 2.5cm, right = 2.5cm, top = 2.5cm, bottom = 2.5cm, headsep = 1.2cm]{geometry}

\usepackage{hyperref}

\title[Fixed point properties for graphical small cancellation]
{Local-to-global fixed point properties for graphical \cftfs and \css small cancellation complexes}

\author{Karol Duda}
\address[K.~Duda]{Silesian University of Technology, Gliwice, Poland}
\email{karol.duda@polsl.pl}

\author{Huaitao Gui}
\address[H.~Gui]{Institut for Matematiske Fag, University of Copenhagen, 2100 Copenhagen, Denmark}
\email{hg@math.ku.dk}

\begin{document}

\maketitle

	\begin{abstract}
    Graphical small cancellation extends the classical small cancellation theory and provides a powerful method for constructing groups with prescribed subgraphs in their Cayley graphs.
    We prove that torsion subgroups of groups defined by possibly infinite \cftfs graphical small cancellation presentations are finite. 
    We also prove the corresponding result for groups defined by \css graphical small cancellation presentations under the additional assumption that the presentation is torsion-essentially \cs-free.
	Both results follow from a more general result on local-to-global fixed point properties for torsion groups acting by automorphisms on simply connected graphical small cancellation complexes.
	\end{abstract}

\section{Introduction}

Graphical small cancellation generalizes classical small cancellation theory. It provides a method for constructing groups whose Cayley graphs contain certain prescribed graphs, giving rise to groups with interesting properties; see, for example, \cite{Gromov-random, OW-Kazhdan, Osajda20, Coulon-Gruber, Osajda-residual}. Recently, graphical small cancellation was also used to show that universal Cayley graphs cannot exist \cite{Tucker-obstructions}.

A graphical presentation of a group has the form $G=\langle f: \Gamma \rightarrow \Theta \rangle$, where $\Theta$ is a connected graph, $f:\Gamma \rightarrow \Theta$ is an immersion, and $G$ is the fundamental group of $\Theta$ modulo the normal subgroup generated by the images of all immersed cycles in $\Gamma$. 
The use of such presentations goes back to the work of Rips and Segev, who constructed torsion-free groups without the unique product property (see \cite{RS-torsion,Steenbock-torsion}). Small cancellation theory over graphical presentations was later used explicitly by Gromov \cite{Gromov-random} in his construction of Gromov's monster---a group that contains an expander graph in a weak sense.

There are two 2-dimensional combinatorial complexes naturally associated to such a presentation: the \textbf{thickened graphical complex} $X_t$, where 2-cells are attached to $\Theta$ to fill the cycles coming from $\Gamma$, and the \textbf{non-thickened graphical complex} $X_c$, where cones over the connected components of $\Gamma$ are attached to $\Theta$ instead. 
We refer to the presentation together with these two realizations as the associated \textbf{graphical complex} $X$. 
We call $\Theta$ the $1$-skeleton of $X$ and denote it by $X^1$. 
Each connected component $\Gamma_i$ of $\Gamma$ determines a \textbf{graphical cell} of $X$. 
Both realizations have fundamental group $G$; in particular, we call $X$ simply connected if either, equivalently both, of its realizations is simply connected. 
Precise definitions will be given in Section~\ref{sec:preliminaries}.

Our main focus is on torsion subgroups of groups defined by graphical small cancellation presentations. 
This is motivated by the general philosophy that graphical small cancellation complexes behave like nonpositively curved spaces, together with several conjectures and questions concerning torsion subgroups and locally elliptic actions on nonpositively curved complexes; see, for example, \cite{Swe99,HaeOsa}.

We prove the following finiteness results.
	
\begin{theorem}\label{thm:tA}
		Torsion subgroups of groups defined by \cftfs graphical small cancellation presentations are finite.
\end{theorem} 

\begin{theorem}\label{thm:tD}
        Not essentially \css torsion subgroups of groups defined by \css graphical small cancellation presentations are finite. In particular, all torsion subgroups of groups defined by torsion-essentially \cs-free graphical small cancellation presentations are finite.
\end{theorem}

The terminology in Theorem~\ref{thm:tD} is meant to isolate the borderline case of the graphical \css condition. 
More precisely, an essentially \css graphical cell is one containing an immersed cycle that can be covered by exactly six pieces; see Definition~\ref{df:essentially}. 
An automorphism is essentially \css if it stabilizes such a cell, and a torsion subgroup is not essentially \css if its canonical action on the associated simplified universal cover is not essentially \css; see Definition \ref{df:essentially_action} and \ref{df:essentially_presentation}. 
Thus Theorem~\ref{thm:tD} leaves open only the case of essentially \css torsion subgroups. 
In particular, this obstruction is absent in the graphical \csvs setting.

    Theorems~\ref{thm:tA} and \ref{thm:tD} are consequences of a more general local-to-global fixed point theorem for actions on simply connected graphical small cancellation complexes. 
    Recall that an action of a group on a metric space is called \textbf{locally elliptic} if every cyclic subgroup has bounded orbits, and \textbf{elliptic} if the whole group has bounded orbits. 
    Haettel and Osajda conjectured that every locally elliptic action of a finitely generated group on a finite-dimensional nonpositively curved complex is elliptic \cite{HaeOsa}. 
    They proved this for several classes of complexes, including locally finite simply connected graphical \cftfs complexes and not necessarily locally finite graphical $C(18)$ complexes.

    For many nonpositively curved complexes, an automorphism fixes a point if and only if the cyclic group generated by it has bounded orbit. However, as we show in Section~\ref{sec:fix}, this equivalence fails for non-locally finite $C(p)$–$T(q)$ small cancellation complexes with $q<5$. 
    Nevertheless, we prove in Section~\ref{sec:fix} that every finite order automorphism of a simply connected \cftfs or \css graphical complex fixes a point.
    Thus torsion group actions in this setting provide a natural class of locally elliptic actions to study.
    
    Our main result is the following local-to-global fixed point theorem. It establishes a particular case of the meta-conjecture of Haettel and Osajda \cite{HaeOsa} for torsion group actions on graphical \cftfs and \css complexes, without assuming local finiteness of the complex or finite generation of the acting group, but under the assumption that the induced action on the $1$-skeleton is free.
   
\begin{theorem}\label{thm:tB}
		Let $X$ be a simply connected \cftfs or \css graphical small cancellation complex. Let $G$ be a torsion group acting on $X$ by automorphisms such that the induced action on the $1$-skeleton $X^{1}$ is free. 
        If $X$ is \cftf, then $G$ is finite.
        If $X$ is \css and the action is not essentially \cs, then $G$ is finite.
        In particular, if $X$ is torsion-essentially \cs-free, then $G$ is finite.
        In either case there is a graphical cell of $X$ stabilized by $G$.
\end{theorem} 

Let us compare Theorem~\ref{thm:tB} with the existing results. 
The known theorems differ along several independent axes: whether the complex is classical or graphical, whether local finiteness is assumed, whether the acting group is finitely generated, and whether the action is free on the $1$-skeleton. 
Duda proved analogous results for classical \cftfs and \css complexes under the assumption that the action is free on the $1$-skeleton \cite{dudaall3}. 
Haettel and Osajda proved ellipticity of locally elliptic actions for finitely generated groups in several graphical settings, including locally finite graphical \cftfs complexes and graphical $C(18)$ complexes \cite{HaeOsa}. 
In the $C(3)$--$T(6)$ case, corresponding results were obtained for finitely generated groups first in the classical setting in \cite{dudaall3} and then in the graphical setting in \cite{Gui}, both using the fixed-point theorem for $2$-dimensional CAT(0) complexes from \cite{NOP-torsion}. 
Theorem~\ref{thm:tB} extends this type of result to not necessarily locally finite graphical \cftfs complexes, and to a subclass of graphical \css complexes containing all graphical \csvs complexes.

The following question remains open.

\begin{question}
   Can an essentially \css torsion subgroup of a group defined by a \css graphical small cancellation presentation be infinite?
\end{question}

The freeness assumption in Theorem~\ref{thm:tB} cannot simply be removed for arbitrary torsion groups. 
Indeed, by a theorem of Serre \cite[Theorem~15, Chapter~I.6.1]{Serre_Trees}, every countable infinitely generated group acts without a fixed point on a tree. 
However, the finitely generated case remains open.

\begin{question}
Does the conclusion of Theorem~\ref{thm:tB} still hold if the freeness assumption on the \(1\)-skeleton is dropped, but the torsion group \(G\) is assumed to be finitely generated?
\end{question}

	Finally, let us present an application of Theorem~\ref{thm:tB} to \textbf{automatic continuity}. This property has its origins in descriptive set theory; roughly, it says that between certain pairs of topological groups, every group homomorphism is automatically continuous. Recently, automatic continuity has been established for a large class of homomorphisms into ``nonpositively curved'' groups, see e.g.\ \cite{keppeler2021automatic}, \cite{BogopolskiCorson}. 

\begin{theorem}
		\label{thm:tC}
		Let $G$ be a group acting geometrically on a simply connected torsion-essentially \cs-free graphical small cancellation complex $X$, and suppose that the action induces a free action on the $1$-skeleton of $X$. If $H$ is a subgroup of $G$, then any group homomorphism $\varphi : L \rightarrow H$ from a locally compact group $L$ is either continuous, or there exists a normal open subgroup $N\subseteq L$ such that $\varphi(N)$ is a finite group.
	\end{theorem}	

    Using Theorem~\ref{thm:tB}, one can also prove an analogous automatic continuity result for groups acting geometrically on simply connected graphical \cftfs complexes. However, a stronger result in this case was already established in \cite{HaeOsa}.

\subsection*{Structure of the paper}
In Section \ref{sec:gsc} we briefly introduce graphical small cancellation theory. 
In Section \ref{sec:sc_graphical} and \ref{sec:quad} we establish basic properties of graphical \cftfs\ complexes.
In Section \ref{sec:scv_graphical} and \ref{sec:sys} we establish basic properties of graphical \css complexes.
Section \ref{sec:cv} introduces the technical tools used in the main lemmas of the paper.

In Section \ref{sec:fix} we prove that every finite-order automorphism of a simply connected graphical \cftfs or \css complex has a fixed point. We also show that the finite-order assumption is necessary by constructing an example of an automorphism with bounded orbits that does not fix any point of the complex.

In Section \ref{sec:main}, we establish necessary conditions under which the concatenation of geodesics remains geodesic. These results are then applied in Section \ref{sec:main2} to derive the main technical lemma needed for the proof of Theorem~\ref{thm:tB}.

Finally, in Section \ref{sec:proofs} we prove Theorems \ref{thm:tA}-\ref{thm:tB}, and in the last section we present an application to automatic continuity.

\section{Preliminaries}\label{sec:preliminaries}

For fundamental notions such as CW-complexes, null-homotopy, and simple connectedness, see Hatcher's textbook \cite{AH}.

A map $X\rightarrow Y$ between complexes is \textbf{combinatorial} if its restriction to every open cell of \(X\) is a homeomorphism onto an open cell of \(Y\). A complex is \textbf{combinatorial} if the attaching map of each cell is combinatorial after a suitable subdivision of the domain sphere. An \textbf{immersion} is a combinatorial map which is locally injective.

Unless stated otherwise, all complexes and maps are combinatorial, and all attaching maps are immersions. We mostly consider $2$-dimensional combinatorial complexes, which we call $2$-complexes.

Starting from Section \ref{sec:quad} we will often use \textbf{square} complexes.
Square complexes are $2$-complexes whose $2$-cells are $4$-gons. In this case, instead of the usual terms, $0$-cell, $1$-cell and $2$-cell, we will use vertex, edge and square, respectively.

A \textbf{graph} is a $1$-dimensional CW-complex with $0$-cells called \textbf{vertices} and $1$-cells called \textbf{edges}. Distances between vertices in graphs are always measured by the \textbf{standard graph metric} which is defined for a pair of vertices $u$ and $v$ as the number of edges in the shortest path connecting $u$ and $v$. 

A graph $\Gamma$ is \textbf{simplicial} if there is no edge in $\Gamma$ with both endpoints attached to one vertex and no two edges of $\Gamma$ having their endpoints attached to the same unordered pair of vertices.
Among classes of simplicial graphs we distinguish the so called \textbf{bipartite} graphs: a graph $\Gamma$ is bipartite if the set of its vertices can be partitioned into two nonempty sets such that no edge has both of its endpoints in the same set.
A graph $\Gamma$ is \textbf{path-cycle extensible} if every immersed path in $\Gamma$ factors through an immersed cycle.

Let $\Gamma$ be a finite simplicial graph. A \textbf{cone} on $\Gamma$ is the quotient space
$$
Cone(\Gamma)=\Gamma\times[0,1]/\Gamma\times\{1\}. 
$$
A \textbf{simplicial cone} is $Cone(\Gamma)$ with a natural structure of a $2$-complex.

A \textbf{link} of a $0$-cell $v$ of a $2$-complex $X$ (denoted $X_v$) is the graph whose vertices correspond to the ends of $1$-cells of $X$ incident to $v$, and an edge joins vertices corresponding to the ends of $1$-cells $e_1, e_2$ if and only if there is a $2$-cell $F$ such that $e_1$ and $e_2$ form a corner of $F$ at $v$.

For $n\in \mathbb{N}$, let $P_n$ denote a \textbf{path graph} with $n$ edges, whose vertices are labeled from $0$ to $n$ in order. A \textbf{path} in a complex $X$ is a map $\gamma: P_n \rightarrow X$ for some $n$. 
The path $\gamma$ is \textbf{closed} if $\gamma(0)=\gamma(n)$. 
A path $\gamma': P_m\rightarrow X$ is called a \textbf{subpath} of $\gamma$ if there exists an injective, order-preserving map $\iota: P_m\rightarrow P_n$ such that $\gamma' = \gamma \circ \iota$. For given paths $P_n\rightarrow X$ and $P'_m\rightarrow X$ such that the terminal point of $P_n$ is equal to the initial point of $P'_m$, their \textit{concatenation} (denoted $P_n\cdot P'_m$) is the natural path $P_n\cdot P'_m\rightarrow X$. 

For $n\in \mathbb{N}_{>0}$, let $C_n$ denote a \textbf{cycle graph} with $n$ edges, whose vertices are labeled by $\mathbb{Z}/n\mathbb{Z}$. A \textbf{cycle} in $X$ is a map $\tau: C_n \rightarrow X$ for some $n$. 
Let $R_n$ denote an $n$-gon, i.e., a disc whose 1-skeleton (called \textbf{boundary}) is $C_n$, denoted $\partial R_n=C_n$. Every 2-cell in $X$ can be represented by a map $R_n\rightarrow X$ for some $n$.
We write $P$, $C$, or $R$ for the domain of a path, a cycle, or a representative of a 2-cell, respectively, when the size is not specified.

Let $\Gamma$ be a simplicial graph where every connected component $\Gamma_i$ of
$\Gamma$ is finite, although $\Gamma$ may have infinitely many
connected components.
Given an immersion of graphs $f:\Gamma \rightarrow \Theta$ with $\Theta$ connected, we define a group $G$ as follows. 
Decompose $\Gamma = \coprod_i \Gamma_i$.
Fix a vertex $u\in \Theta$, and for each $i$, choose a vertex $u_i\in \Gamma_i$ and a path $\gamma_i$ in $\Theta$ from $u$ to $f(u_i)$. Define a homomorphism $f_i: \pi_1(\Gamma_i,u_i)\rightarrow \pi _1(\Theta,u)$, $[\gamma]\mapsto [\gamma_i\cdot f\circ\gamma \cdot \overline{\gamma_i}]$, where $\gamma$ is a closed path in $\Gamma_i$ based at $u_i$, $\cdot$ denotes path concatenation, and $\bar{\cdot}$ denotes path reversal. 
Then define $G \coloneqq \pi_1(\Theta,u)\big/ 
\left\langle\!\left\langle \bigcup_i f_i(\pi_1(\Gamma_i,u_i)) \right\rangle\!\right\rangle$. This group is well-defined up to isomorphism, independently of the choices of the basepoints and connecting paths. 
We call $f:\Gamma\rightarrow \Theta$ a \textbf{graphical presentation} of $G$, denoted $G=\langle f:\Gamma \rightarrow \Theta \rangle$.
A graphical presentation is \textbf{simplified} if for $i\neq j$, there is no isomorphism $\Gamma_i \rightarrow \Gamma_j$ such that the following diagram commutes:

\[
\begin{tikzcd}
    \Gamma_i \arrow[rd] \arrow[d] & \\
    \Gamma_j \arrow[r] & \Theta
\end{tikzcd}
\]
For the rest of the paper, unless stated otherwise, we always assume that our graphical presentations are simplified. 
There is no loss of generality, as for any graphical presentation, there exists a simplified one defining the same group.

When $\Theta$ is a wedge of circles, and each $\Gamma_i$ is a cycle graph, the graphical presentation reduces to a classical group presentation. One can then associate a 2-complex to such a presentation by attaching a 2-cell to $\Theta$ for each $\Gamma_i\rightarrow \Theta$, this complex is referred to as the \textbf{presentation complex}. 
Note that each 2-cell in this complex can also be thought of as a cone over its boundary cycle $\Gamma_i$. In the classical setting, this change does not modify the underlying geometric structure significantly.

In the general setting, these two viewpoints of what a presentation complex is, lead to two natural ways of constructing a 2-complex with fundamental group $G=\langle f:\Gamma \rightarrow \Theta \rangle$. First, let $\Theta$ be the 1-skeleton. For every immersed cycle $C\rightarrow \Gamma$, attach a 2-cell to $\Theta$ along the composition $C\rightarrow \Gamma \rightarrow \Theta$. The resulting complex is called the \textbf{thickened graphical complex} associated to $G$, and we denote it by $X_t$. 
The term “thickened” comes from the fact, that for any connected component $\Gamma_i$,
we have a \textbf{thick cell} which is formed by gluing 2–cells along all immersed
cycles in $\Gamma_i$. Alternatively, for each connected component $\Gamma_i$ of $\Gamma$, attach the simplicial cone $Cone(\Gamma_i)$ to $\Theta$ via the map $f|_{\Gamma_i}$. The resulting complex is called the \textbf{non-thickened graphical complex}, denoted $X_c$. Note that every 2-cell in $X_c$ is a triangle.

Observe that in the general setting, the difference between the topology of thickened and non-thickened graphical complexes becomes significant, as e.g. thickened cells may contain spheres, while non-thickened ones cannot. Consequently, these complexes are used to obtain different properties of a presentation. 

Note that the thickened and non-thickened graphical complexes associated to a graphical presentation may have interesting properties even when the associated group itself is not particularly interesting. This is especially relevant in the case of simply connected thickened or non-thickened graphical complexes.
Consequently, our main object of interest is not always the group defined by the graphical presentation, but rather the associated \textbf{graphical complex} $X$. By this we mean the graphical presentation together with its two realizations: the thickened graphical complex $X_t$ and the non-thickened graphical complex $X_c$. From now on, we refer to these as the thickened and non-thickened realizations, respectively.
In these complexes, a \textbf{graphical cell} is defined by a connected component $\Gamma_i$ of $\Gamma$, together with its realization as a thick cell in the thickened graphical complex and as a cone in the non-thickened graphical complex.

\begin{rem}
\label{r:2-complex}
    For any $2$-complex $Y$ there is an associated graphical complex $X$ obtained by taking $\Theta$ as the $1$-skeleton $Y^{(1)}$ and $\Gamma$ as a union of cycle graphs $\Gamma_i$, one for each $2$-cell of $Y$.
    Note that if we think of each $2$-cell in $Y$ as a cone over its boundary cycle, then we obtain exactly the non-thickened graphical complex $X_c$. 
    On the other hand, a thick cell $\mathrm{Th}(F)$ associated to a $2$-cell $F$ in $Y$ is the union of $F$ and the set of $2$-cells attached along coverings of $\partial F$.
    Some specialists in the area believe that the thickened graphical complex $X_t$ can be defined in a way that identifies $X_t$ with $Y$, but the current definition significantly simplifies the proofs of basic properties.
\end{rem}

From now on, we will assume that $\Gamma$ is path-cycle extensible.
There is no loss of generality because we can always replace $\Gamma$ by the subgraph $\Gamma'$ given by the union of all immersed cycles in $\Gamma$. 
The graphical presentations $\langle f: \Gamma \rightarrow \Theta\rangle $ and $\langle f|_{\Gamma'}: \Gamma' \rightarrow \Theta \rangle$ define the same group, the same thickened graphical complex, while their non-thickened graphical complexes are homotopy equivalent.

Let $X$ be a graphical complex associated to a graphical presentation $G=\langle f:\Gamma \rightarrow \Theta \rangle$. We describe the graphical structure of its \textbf{universal cover} $\tilde{X}$. Fix an orientation of every edge in $\Theta$, for every edge $e\in E(\Theta)$, let $o(e)$ be its initial vertex, and $t(e)$ be its terminal vertex. Let $E_o(\Theta)$ be the collection of all oriented edges in $\Theta$, i.e. $E_o(\Theta):=\{(o(e),t(e)):e\in E(\Theta)\}$. Fix a spanning tree $T$ of $\Theta$.  Recall that every oriented edge $s\in E_o(\Theta)\setminus E_o(T)$ corresponds to a nontrivial element $[s]\in G$, and all such elements generate the whole $G$. Define a graph $\tilde{\Theta}$ as follow. Let $V(\tilde{\Theta}):=V(\Theta)\times G$ be its vertex set, and $E_o(\tilde{\Theta}):=E_o(\Theta)\times G$ be its edge set.  For every edge $e\in E_o(T)$, let the initial vertex of $(e,g)\in E_o(\tilde{\Theta})$ be $(o(e),g)\in V(\tilde{\Theta})$ and the terminal vertex be $(t(e),g) \in V(\tilde{\Theta})$ respectively. For every edge $s\in E_o(\Theta)\setminus E_o(T)$, define $o((s,g)):=(o(s),g)$ and $t((s,g)):=(t(s),g[s])$. Observe that $\tilde{\Theta}$ is a covering space of $\Theta$. Let $\tilde{\Gamma}=\Gamma \times G$. For each $i$, fix a vertex $v_i\in \Gamma_i$. Let $\tilde{f}:\tilde{\Gamma}\rightarrow \tilde{\Theta}$ be the map where the restriction $\tilde{f}|_{\Gamma_i\times \{g\}}$ is the unique lift of $f|_{\Gamma_i}$ that sends $v_i$ to $(f(v_i), g)$. 
Since $f$ is immersion of graphs, $\tilde{f}$ is also an immersion and $\langle \tilde{f}:\tilde{\Gamma} \rightarrow \tilde{\Theta} \rangle$ is a graphical presentation.
We define $\tilde{X}$ as the graphical complex associated to the graphical presentation $\langle \tilde{f}:\tilde{\Gamma} \rightarrow \tilde{\Theta} \rangle$. Note that the canonical maps $\pi_t:\tilde{X_t}\rightarrow X_t$ and $\pi_c:\tilde{X_c}\rightarrow X_c$ are covering maps.

\begin{pr}
    $\tilde{X_t}$, $\tilde{X_c}$ are simply connected.
\end{pr}

\begin{proof}
    Since $\tilde{X}_t$ and $\tilde{X}_c$ have the same fundamental group, it suffices to show that any closed path in $\tilde{X}_t$ is homotopically trivial.
    Fix a vertex $v\in V(\Theta)$ and $v_0=(v,g)\in \pi^{-1}(v)$. Let $\tilde{p}: C\rightarrow \tilde{\Theta}$ be a closed path in $\tilde{\Theta}$ starting at $v_0$, and let $p:=\pi\circ \tilde{p}$. Without loss of generality, assume the closed path $p$ represents the element $[s_1]\cdots[s_n]$ in $\pi_1(\Theta,v)$, where $s_i \in E_o(\Theta)$. Fix $p_i$ as a path which goes from $v$ to $o(s_i)$ in $T$, from $o(s_i)$ to $t(s_i)$ along $s_i$, and from $t(s_i)$ to $v$ in $T$. Then $p$ is homotopically equivalent to the concatenation $p_1\cdots p_n$. By the unique lifting property, there is a unique lift $\tilde{p}_1$ of $p_1$ starting from $v_0$ and ending at $v_1:=(v,g[s_1])$. Lift $p_i$ inductively, as a path from $v_{i-1}$ to $v_i$, then $v_n=(v,g[s_1]\cdots[s_n])$. Since $p$ is homotopically equivalent to $p_1\cdots p_n$, $\tilde{p}$ is homotopically equivalent to $\tilde{p_1}\cdots \tilde{p_n}$, in particular, $v_n=v_0$, i.e., $[s_1]\cdots[s_n]$ is trivial. Hence, $p$ is homotopically trivial in $X_t$, and since $\tilde{X_t}$ is a covering space, $\tilde{p}$ is also homotopically trivial. This concludes the proof that $\tilde{X_t}$ is simply connected.
\end{proof}

Note that the presentation $\langle\tilde{f}: \tilde{\Gamma}\rightarrow \tilde{\Theta}\rangle$ of the universal cover $\tilde{X}$ is not necessarily simplified, we denote by $\langle\tilde{f}: \tilde{\Gamma}^{+}\rightarrow \tilde{\Theta}\rangle$ its simplification, and call the associated graphical complex $\tilde{X}^{+}$ the \textbf{simplified universal cover} of $X$. It is clear that both $\tilde{X}^{+}_t,\tilde{X}^{+}_c$ are also simply connected.

\section{Graphical small cancellation}\label{sec:gsc}
In this section, we first recall some basic definitions related to small cancellation, then we define small cancellation conditions for graphical complexes.
We use definitions that are slightly modified from those in, for example, \cite{OsajdaPrytula,Helly}. 

The following definition is crucial in small cancellation theory.
A \textbf{piece} in a 2-complex $Y$ is a path $P\rightarrow Y$ that factors in two essentially distinct ways through 2-cells, i.e., there are 2-cells $R\rightarrow Y$ and $R'\rightarrow Y$ such that $P\rightarrow Y$ factors both as $P\rightarrow \partial R \rightarrow Y$ and $P\rightarrow \partial R' \rightarrow Y$, and there is no isomorphism $\partial R\rightarrow \partial R'$ making the following diagram commute:
\[
\begin{tikzcd}
    P \arrow[r] \arrow[d] & \partial R \arrow[d] \arrow[ld]\\
    \partial R' \arrow[r] & Y
\end{tikzcd}
\]

Similarly, a \textbf{piece} in a graphical complex $X$ is also an immersed path $\gamma: P\rightarrow \Theta$ which can factor through graphical cells in two essentially different ways, that is, $\gamma$ can factor as $P\rightarrow \Gamma_i \rightarrow \Theta$ and $P\rightarrow \Gamma_j \rightarrow \Theta$, and there is no isomorphism $\Gamma_i \rightarrow \Gamma_j$ such that the following diagram commutes:
\[
\begin{tikzcd}
    P \arrow[r] \arrow [d] & \Gamma_i \arrow[d]\arrow[ld] \\ 
    \Gamma_j \arrow[r] & \Theta
\end{tikzcd}
\]

\begin{rem}
There are two different ways of generalizing the definition of pieces to graphical complexes (see \cite[Definition 1.2, Definition 1.3]{Gruber-graphical}).
Here we use the original definition by Gromov \cite{Gromov-random}. 
Unlike the other definition, it allows groups defined by graphical small cancellation presentation to have torsion, which is relevant for our results.
\end{rem}

In the classical case, we are interested in combinatorial maps between $2$-complexes, specifically, the so called \textbf{reduced}
\footnote{The definition given here was suggested by a referee of \cite{Helly} in the graphical setting. When restricted to the classical setting, it is equivalent to the definition given in \cite{MW}.} 
maps. 
A map between $2$-complexes $Y$ and $X$ is \textbf{reducible} along a piece $P\rightarrow Y$ if $P\rightarrow Y$ can factor through $2$-cells in two essentially distinct ways as $P\rightarrow \partial R_Y\rightarrow Y$ and $P\rightarrow \partial R'_Y\rightarrow Y$, while the corresponding factorizations of $P\rightarrow X$ as $P\rightarrow R_Y \rightarrow R_X \rightarrow X$ and $P\rightarrow R'_Y \rightarrow R'_X\rightarrow X$ are essentially the same, i.e., there is no isomorphism $\partial R_Y\rightarrow \partial R'_Y$ making the outer diagram commute, while there is an isomorphism $\partial R_X\rightarrow \partial R'_X$ making the inner diagram commute:

\[
\begin{tikzcd}
    P \arrow[rd, equal] \arrow [rrr] \arrow[ddd] &  &  &  \partial R_Y \arrow[ddd]\arrow[ld] \\
     & P \arrow[r] \arrow [d] & \partial R_X  \arrow[d]\arrow[ld] &  \\ 
     & \partial R'_X \arrow[r] & X & \\
    \partial R'_Y \arrow[rrr] \arrow[ur] &  &   & Y \arrow[ul]
\end{tikzcd}
\]
A map $Y\rightarrow X$  is \textbf{reduced} if it is not reducible along any piece $P\rightarrow Y$.

From now on, we consider combinatorial maps from $2$-complexes to graphical complexes. A map from a $2$-complex naturally targets either a thickened or a non-thickened realization, and only rarely can the same complex be mapped to both. In this paper, we focus on maps to thickened graphical complexes, since this setting is necessary for formulating $T(q)$ small cancellation condition and for obtaining useful connections between classical and graphical small cancellation theory.

A map $Y\to X_t$ from a $2$-complex $Y$ to a thickened graphical complex $X_t$ is \textbf{reducible} along a piece $P\rightarrow Y$, if $P\rightarrow Y$ can factor in two essentially different ways as $P\rightarrow \partial R\rightarrow Y$ and $P\rightarrow \partial R'\rightarrow Y$, with the corresponding factorizations of $P\rightarrow X_t$ as $P\rightarrow R \rightarrow \mathrm{Th}(\Gamma_{i})\rightarrow X_t$ and $P\rightarrow R' \rightarrow \mathrm{Th}(\Gamma_{j})\rightarrow X_t$ being essentially the same, i.e., there is no isomorphism $\partial R\rightarrow \partial R'$ such that the outer diagram commutes, yet there is an isomorphism $\Gamma_{i}\rightarrow \Gamma_{j}$ making the inner diagram commute:

\[
\begin{tikzcd}
    P \arrow[rd, equal] \arrow [rrr] \arrow[ddd] &  &  &  \partial R \arrow[ddd]\arrow[ld] \\
     & P \arrow[r] \arrow [d] & \Gamma_{i} \arrow[d]\arrow[ld] &  \\ 
     & \Gamma_{j} \arrow[r] & X_t & \\
    \partial R' \arrow[rrr] \arrow[ur] &  &   & Y \arrow[ul]
\end{tikzcd}
\]
A map $Y\rightarrow X_t$  is \textbf{reduced} if it is not reducible along any piece $P\rightarrow Y$. 

A \textbf{spherical diagram} $D$ is a $2$-sphere $\mathbb{S}^2$ with a structure of a finite combinatorial $2$-complex.
A \textbf{disc diagram} $D$ is a contractible subspace of the plane $\mathbb{R}^2$ with the structure of a finite combinatorial $2$-complex. We say that a disc diagram $D$ is \textbf{nonsingular} if it is homeomorphic to a closed $2$-cell and \textbf{singular} otherwise. 
The \textbf{boundary} $\partial D$ of a disc diagram $D$ is the attaching map of the $2$-cell that contains the point $\infty$ when we regard $\mathbb{S}^2=\mathbb{R}^2\cup \{\infty\}$.
A \textbf{spur} is an edge in the boundary $\partial D$ that has a vertex of degree 1.
If a disc diagram $D$ is singular, then it is either trivial, or it consists of a single $1$-cell joining two $0$-cells, or it contains a \textbf{cut} $0$-\textbf{cell (or cut-vertex)}, i.e.~a $0$-cell $v$ of $D$ such that $D\setminus\{v\}$ is disconnected. Components of $D$ with all cut-vertices removed are called \textbf{cut-components}.
A \textbf{diagram in} a 2-complex $Y$ is a map from such a spherical or disc diagram $D$ to $Y$.
A \textbf{diagram in} a graphical complex $X$ is a map from a spherical or disc diagram $D$ to either the thickened complex $X_t$ or the non-thickened complex $X_c$. 

A diagram $D$ in a $2$-complex $Y$ is called \textbf{reduced} if the map $D\to Y$ is reduced.
Analogously a diagram $D$ in a graphical complex $X$ is called \textbf{reduced} if the map $D\to X_t$ is reduced.

 \begin{lm}[Lyndon-van Kampen Lemma]\label{l:vank}
 	Let $X$ be a graphical complex and let $C \to \Theta$ be a closed path which is null-homotopic in $X_t$. Then
 	\begin{enumerate} 
 		
 		\item \label{l:vk1} there exists a disc diagram $D\to X_t$ such that the path $C$ factors as $C \to \partial D \to X_t$, and $C \to \partial D$ is an isomorphism,
 		
 		\item \label{l:vk2} if a diagram $D \to X_t$ is not reduced, then there exists a diagram $D_1 \to X_t$ with smaller area and the same boundary in the sense that there is a commutative diagram:
 		
\[
    \begin{tikzcd}
    \partial D_1 \arrow[r,"\cong"] \arrow[rd] & \partial D \arrow[d] \\
     & X_t
    \end{tikzcd}
\]
 		
 		\item \label{l:vk3} any minimal area diagram $D \to X_t$ such that $C$ factors as $C \xrightarrow{\cong} \partial D \to X_t$ is reduced.
 	\end{enumerate}
 \end{lm}

\begin{pr}\label{reducible_spherical}
    Let $S$ be a spherical diagram in $X_t$ and let $C \to S$ be an injective nontrivial closed path. Then two subcomplexes $D_1,D_2$ of $S$ bounded by $C$ are disc diagrams in $X_t$.

    If $S$ is reduced in $X_t$ then $D_1,D_2$ are reduced. If $S$ is reducible, then there exists $C$ that either one of $D_1,D_2$ is reducible, or $D_1$ and $D_2$ both consist of a single cell, mapped to the same thick cell of $X_t$.
\end{pr}

\begin{proof}
    By definition of spherical and disc diagrams, it is clear that $D_1,D_2$ are disc diagrams in $X_t$ for any such path $C$. By definition of reduced diagram, it is also clear that $S$ being reduced implies that $D_1,D_2$ are reduced.

    If $S$ is reducible along a piece $P$, and there exists an injective closed path $C\to S$ that does not contain $P$, then one of $D_1,D_2$ contains $P$ as a piece, thus it is reducible.
    If every injective closed path $C\to S$ contains $P$, then $S$ has no injective closed path avoiding $P$. Hence $S$ consists of two $2$-cells glued along their whole boundary. Since the reducibility occurs along $P$, these two cells map to the same thick cell of $X_t$.
\end{proof}

\begin{df}
    For $p,q\in \mathbb{N}^+$, we say the graphical complex $X$ associated to a graphical presentation $G=\langle f:\Gamma \rightarrow \Theta\rangle$ satisfies 
    \begin{itemize}
        \item the \textbf{C(p) small cancellation condition} if there is no immersed cycle $C\rightarrow\Gamma$ such that $C\rightarrow \Gamma \rightarrow \Theta$ is the concatenation of less than $p$ pieces;
        \item the \textbf{T(q) small cancellation condition} if for every reduced disc diagram $D\rightarrow X_t$, there is no interior vertex of $D$ whose degree is greater than 2 and less than $q$. 
    \end{itemize}  
    In particular $C(p)$--$T(q)$ denotes complex or presentation that satisfies both $C(p)$ and $T(q)$. 
\end{df}

\begin{df}\label{df:essentially}
    Let $X$ be a $C(p)$ graphical complex, and $\Gamma_i$ be a graphical cell.
    We say that $\Gamma_i$ is \textbf{essentially} $C(p)$ if there is an immersed cycle $C\rightarrow\Gamma_i$ such that $C\rightarrow \Gamma_i \rightarrow \Theta$ is the concatenation of exactly $p$ pieces. 
\end{df}

Thus, in a $C(p)$ complex, an essentially $C(p)$ cell is one in which the $C(p)$ lower bound is sharp.

A 2-complex $Y$ is $C(p)$--$T(q)$ if its associated graphical presentation (see Remark \ref{r:2-complex}) is.  

\begin{pr}\label{reduceddiagram}
  If $X$ is a $C(p)$--$T(q)$
  graphical complex and $D \to X_t$ is a reduced disc or spherical diagram, then $D$
  is a $C(p)$--$T(q)$ diagram.
\end{pr}

\begin{proof}
By the definition of reduced diagram, each piece in $D$ maps to a piece in $X_t$, consequently $D$ has to satisfy $C(p)$.
Moreover if $D'\to D$ is a reduced disc diagram, then $D'\to D \to X_t$ is also reduced, thus $D$ has to satisfy the same $T(q)$ as $X_t$.
\end{proof}

\begin{lm}
 If $X$ is a $C(p)$--$T(q)$
  graphical complex then its simplified universal cover $\tilde{X}^{+}$ is also a $C(p)$--$T(q)$
  graphical complex.
\end{lm}

\begin{proof}
    We show that every pair of essentially different factorizations of a piece $P\rightarrow \tilde{X}^+_t$ induces a pair of essentially different factorizations of the path $P\rightarrow \tilde{X}^+_t \xrightarrow{\pi} X_t$. Consequently, under the map $\pi: \tilde{X}^+_t \rightarrow X_t$, every piece is mapped to a piece, and every reduced disc diagram is again mapped to a reduced disc diagram. Therefore, $X_t$ is $C(p)$--$T(q)$ implies that $\tilde{X}^+_t$ is $C(p)$--$T(q)$.

    For $i\in\{1,2\}$, let $P\rightarrow \tilde{\Gamma}_i\rightarrow \tilde{\Theta}$ be two factorizations of a piece $P\rightarrow \tilde{X}_t^+$ such that there is no isomorphism $\tilde{\Gamma}_1\rightarrow \tilde{\Gamma}_2$ making the following the diagram commute:
    \[
    \begin{tikzcd}
       P \arrow[r] \arrow [d] & \tilde{\Gamma}_1 \arrow[d]\arrow[ld] \\ 
        \tilde{\Gamma}_2 \arrow[r] & \tilde{\Theta}
    \end{tikzcd}
    \]

    Note that the two factorizations of the piece $P\rightarrow\tilde{X}_t^+$ naturally induce factorizations of the path $P\rightarrow \tilde{X}^+_t\xrightarrow{\pi} X_t$ as $P\rightarrow \tilde{\Gamma}_i\rightarrow \tilde{\Theta}\rightarrow \Theta$ for $i\in\{1,2\}$. If the induced factorizations are not essentially different, i.e., there is an isomorphism $\tilde{\Gamma}_1\rightarrow \tilde{\Gamma}_2$ such that the following diagram commutes:
    \[
    \begin{tikzcd}
       P \arrow[r] \arrow [d] & \tilde{\Gamma}_1 \arrow[d]\arrow[ld] \\ 
        \tilde{\Gamma}_2 \arrow[r] & \Theta
    \end{tikzcd}
    \]
    Then by the unique lifting property of the covering map $\tilde{\Theta}\rightarrow \Theta$, the composition $\tilde{\Gamma}_1\rightarrow \tilde{\Gamma}_2\rightarrow \tilde{\Theta}$ and the map $\tilde{\Gamma}_1\rightarrow \tilde{\Theta}$ must also be the same, therefore, the following commuting diagram: 
    \[
    \begin{tikzcd}
       P \arrow[r] \arrow [d] & \tilde{\Gamma}_1 \arrow[d]\arrow[ld]\arrow[rdd] & \\ 
        \tilde{\Gamma}_2 \arrow[r]\arrow[rrd] & \Theta &\\
        & & \tilde{\Theta} \arrow[ul]
    \end{tikzcd}
    \]
    However, this contradicts our assumption that the factorizations $P\rightarrow \tilde{\Gamma}_i\rightarrow \tilde{\Theta}$ are essentially different. 
    Therefore, the induced factorizations $P\rightarrow \tilde{\Gamma}_i\rightarrow \tilde{\Theta}\rightarrow \Theta$ must also be essentially different.

\end{proof}

The same argument also shows that the ordinary universal cover is $C(p)$--$T(q)$. We focus on the simplified universal cover because this will be important in the proofs of Theorems~\ref{thm:tA} and~\ref{thm:tD}.

We will also need the following terminology. 

\begin{df}
\label{df:essentially_action}
    Let $X$ be a $C(p)$ graphical complex. 
    An automorphism $f\in \mathrm{Aut}(X)$ is called \textbf{essentially} $C(p)$ if $f$ stabilizes an essentially $C(p)$ cell of $X$. 
    An action of a group $G$ on $X$ by automorphisms is called \textbf{essentially $C(p)$} if every element of $G$ is essentially $C(p)$. 
\end{df}

\begin{df}
\label{df:essentially_presentation}
Let $G$ be a group defined by a $C(p)$ graphical presentation, and let $\tilde{X}^{+}$ be the associated simplified universal cover. A subgroup $H\leq G$ is called \textbf{not essentially $C(p)$} if the canonical action of $H$ on $\tilde{X}^{+}$ is not essentially $C(p)$.

A $C(p)$ graphical complex $X$ is called \textbf{torsion-essentially $C(p)$-free} if every finite-order automorphism of $X$ whose action on $X^1$ is free is not essentially $C(p)$. 
A $C(p)$ graphical presentation is called \textbf{torsion-essentially \(C(p)\)-free} if its associated simplified universal cover $\tilde{X}^{+}$ is torsion-essentially $C(p)$-free.
\end{df}

\section{Simply connected $C(4)$--$T(4)$ graphical complexes}\label{sec:sc_graphical}

In this section, we recall several properties of simply connected \cftfs graphical complexes proved in \cite{Helly}. 
The statements of the lemmas below are the same as those in \cite{Helly}. 
However, \cite{Helly} uses a slightly different notion of reduced disc diagram. For this reason, we reproduce the proofs, following the arguments of \cite{Helly}, with the necessary adjustments and some additional details.

The following lemma is \cite[Lemma 6.15]{Helly}. It is the graphical \cftfs analogue of \cite[Theorem 6.10]{OsajdaPrytula} in the graphical $C(6)$ setting, and of \cite[Propositions 3.4, 3.5, 3.7 and Corollary 3.6]{HodaQuadric} in the classical \cftfs setting.

\begin{lm}
\label{lem:intersections}
Let  $X$ be a simply connected \cftfs graphical complex with associated presentation $\langle f :\Gamma \rightarrow \Theta \rangle$.
Then the following hold:
  \begin{enumerate}
  \item  For every graphical cell $\Gamma_i$, the map
    $\Gamma_i \to \Theta$ is an embedding.
  \item The intersection of (the images of) any two
    graphical cells is either empty or it is a finite tree.
  \item If three graphical cells pairwise intersect then
    they triply intersect and the intersection is a finite tree.		
  \end{enumerate}
\end{lm}

\begin{proof}
The proofs of all the items (1), (2), (3) follow the same lines: we assume the statement does not hold and we show that this implies existence of a reduced disc diagram contradicting \cftf.

(1) Suppose there is a graphical cell $\Gamma_1$ that does not embed. Since we assume that $\Gamma$ is simplicial, there is a path $\gamma\in \Gamma_1$ with ends $v\neq v'\in$ $\Gamma_1$ such that both $v$ and $v'$ are mapped to a common vertex $v_{11}$ in $X_t$ and $\gamma$ is mapped to a loop $\gamma_{1}$ in $X_t$. 
By simple connectedness and by Lemma~\ref{l:vank}
there exists a reduced disc diagram $D$ for $\gamma_{1}$, see the left part of Figure~\ref{f:c4t4A}. We may assume that we choose $v,v',\gamma,D$  so
that the number of all (0-,1-,2-) cells of $D$ is minimal among all possible choices. 

Note that this choice ensures that $\gamma$ is immersed. 
Indeed, if $\gamma$ is not immersed then it has a spur $S$, and therefore, $\gamma_1$ also has a spur $S_1$. Then disc diagram $D'$, bounded by $\gamma_1\setminus S_1$ has two less cells. 
Then we can choose path $\gamma'=\gamma-S$ mapped to $\gamma_1\setminus S_1$, a contradiction with minimality of choice of $v,v',\gamma$.

Moreover, $\gamma_1$ is also immersed.
Indeed, if $\gamma = (u_0 =v, ..., u_n =v')$ is immersed, and $\gamma_1$ is not. Since $\Gamma \rightarrow \Theta$ is an immersion, $(u_0,u_1)$ and $(u_n,u_{n-1})$ must be mapped to the same edge of $\gamma_1$ and $u_1\neq u_{n-1}$. We can then take $\gamma' = (u_1,...,u_{n-1})$, and $\gamma'_1$ be the image of it. It is clear there is a disc diagram bounded by $\gamma'_1$ that has two less cells than $D$, a contradiction. 

Furthermore, $D$ does not have cut vertices and is non-singular, therefore it contains at least one $2$-cell.
First, observe that $v_{11}$ is not a cut vertex. If it is then there exists $0\leq i< n$ such that $u_i$ is mapped to $v_{11}$. Take $\gamma' = (u_0,u_1,...,u_{i})$, clearly there is a disc diagram bounded by $\gamma'_1$ that has less cells than $D$, a contradiction.

Consequently there is only a single cut component $S_v$ of $D$ that contains $v_{11}$ in its boundary.
Let $S\neq S_v$ be a cut component of $D$ with only a single cut vertex. Since $D$ is a disc diagram, such component has to exist. Let $\gamma_S$ be a subpath of $\gamma_1$ corresponding to the boundary of $S$. Clearly there exists a disc diagram bounded by $\gamma_S$ with less cells than $D$.

Take $\gamma'= (u_i,...,u_{j})$ whose image is $\gamma_S$. 
If $u_i\neq u_j$, then $\gamma'$ is a path that gives us smaller disc diagram, a contradiction.
If $u_i= u_j$, then $\gamma\setminus\gamma'$ is a path that gives us smaller disc diagram, again a contradiction. 

\begin{figure}[h]
  \centering
  \includegraphics[width=1\textwidth]{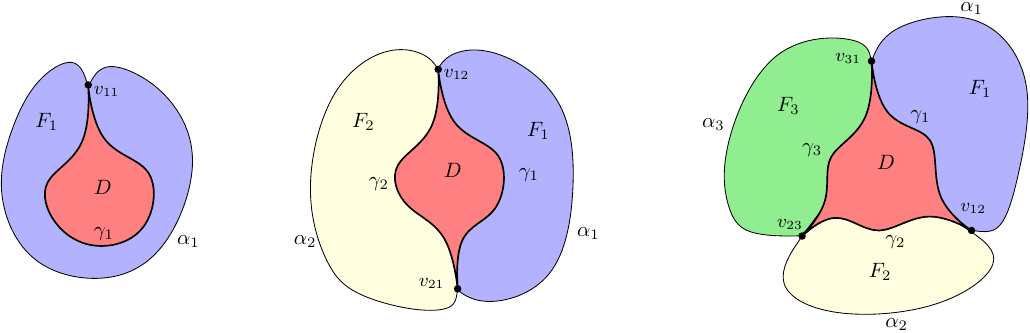}\caption{The proof of Lemma~\ref{lem:intersections}. From left to right: (1), (2), (3).}\label{f:c4t4A}
\end{figure}
	
Now, since $\Gamma$ is path-cycle extensible, there exists a 2-cell $F_1$ whose boundary is the concatenation $\gamma_1\alpha_1$ which is mapped to a loop in $\Gamma_1$. 
Consider a larger disc diagram $D\cup F_1$, and assume that in this disc the only common point of $\gamma_1$ and
$\alpha_1$ is $v_{11}$, see left part of Figure~\ref{f:c4t4A}. Such a disc cannot be reduced, since
otherwise by Proposition~\ref{reduceddiagram} it would be a \cftfs diagram, and this would contradict e.g.\
\cite[Proposition 3.4]{HodaQuadric}. Hence, by the definition of a
reduced diagram, $D\cup F_1$ is reducible along some piece $P$ in $D\cup F_1$. 
Since there are no pieces along $\alpha_1$ and $D$ is reduced, it follows that the piece $P$ has to lie on $\gamma_1$. 
Since $D\cup F_1$ is reducible along $P$
in $X_t$, $P$ factors through $P\to \partial F_1\to \Gamma_1\to X_t$ and some cell $P\to \partial F'\to \Gamma_m\to X_t$, for some cell $F'$, and there is an isomorphism between $\Gamma_m$ and $\Gamma_1$ such that the following diagram commutes
\[
    \begin{tikzcd}
    P\arrow[r] \arrow[rd] & \partial F' \arrow[r] & \Gamma_m \arrow[d,"\cong"] \arrow[rd]  \arrow[d] \\
    &\partial F_1 \arrow[r] & \Gamma_1 \arrow[r] & X_t
    \end{tikzcd}
\]
Thus $\partial F'$ factors through $\Gamma_1$. Consequently $Q=\partial F'\setminus P$ is a path in $\Gamma_1$ see Figure~\ref{f:c4t4C}. 
Take a subpath of $\gamma$ corresponding to $P$. And take a path in $\Gamma_1$ corresponding to $Q$. Take $\gamma'=(\gamma\setminus P)\cup Q$. Since $P$ is a piece in $D\cup F_1$, by definition, $P\rightarrow \partial F_1\rightarrow \Gamma_1$ has to be a subpath of $\gamma$. Therefore $\gamma'$ contains $v$ and $v'$. For corresponding $\gamma'_1$ we get a disc diagram $D'$, such that $D=D'\cup F'$. Clearly $D'$ has smaller number of cells, a contradiction.

\begin{figure}[h]
  \centering
  \includegraphics[width=0.35\textwidth]{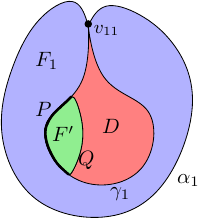}\caption{The proof of Lemma~\ref{lem:intersections}(1).}\label{f:c4t4C}
\end{figure}

(2) First we prove that the intersection of two graphical cells is connected. We proceed analogously to the proof of (1). Suppose not, and let $\Gamma_1,\Gamma_2$ intersect in a non-connected subgraph leading to a reduced disc
diagram as in the middle of Figure~\ref{f:c4t4A}, with the boundary of $F_i$ mapping to $\Gamma_i$. 

Take vertices $v_{12},v_{21}$ in distinct components of the intersection of $\Gamma_1\cap \Gamma_2$. Let $\gamma(i)$ be a path joining $v_{12}$ and $v_{21}$ in $\Gamma_i$. Let $\gamma_i$ be an image of $\gamma(i)$ in $X_t$. It is clear that $\gamma_1\cup\gamma_2$ is a loop in $X_t$.
Analogously to the proof of case (1), we assume that we choose $v_{12},v_{21},\gamma(1),\gamma(2)$ so that resulting disc diagram $D$ has minimal number of all (0-,1-,2-) cells among all possible choices. 
We consider the extended disc diagram $D\cup F_1 \cup F_2$, where $F_i$ is the 2-cell corresponding to some extension of $\gamma(i)$ in $\Gamma_i$. By \cite[Proposition 3.5]{HodaQuadric} the new diagram is not reduced and hence,
as in the proof of (1) we get to a contradiction by finding a new counterexample with a smaller area diagram.
This proves the connectedness of the intersection of two graphical cells.

The fact that such intersections does not contain cycles follows immediately from the $C(4)$ condition.

(3) By (1) and (2) it is enough to show that the triple intersection is non-empty.
Here we proceed analogously to (1) and (2). The corresponding diagrams are depicted in Figure~\ref{f:c4t4A} on the right, and the fact that the extended diagram $D\cup F_1 \cup F_2 \cup F_3$ is not reduced follows
from \cite[Proposition 3.7]{HodaQuadric}.
\end{proof}

We next recall the \textbf{strong Helly property} and the \textbf{Helly property} for simply connected graphical \cftfs complexes. These are \cite[Lemma 6.16 and Lemma 6.17]{Helly}, stated in our notation. 

\begin{lm}[Strong Helly property]\label{l:c4t4strongH}
  Let $\Gamma_1,\Gamma_2,\Gamma_3$ be three pairwise intersecting graphical cells in a simply connected \cftfs graphical complex.
  Then there is an intersection $\Gamma_i\cap \Gamma_j$ of two graphical cells which is contained in the
  third one.
\end{lm}

\begin{proof}
We may assume that the three graphical cells are distinct, since otherwise the conclusion is immediate. Suppose that no intersection of two of them is contained in the third. Then, for each $\{i,j,k\}=\{1,2,3\}$, choose a vertex $v_i\in \Gamma_j\cap \Gamma_k\setminus \Gamma_i$.
  By Lemma~\ref{lem:intersections} there
  exists a vertex $v\in \Gamma_1\cap \Gamma_2 \cap \Gamma_3$ and immersed paths
  $\gamma_i\subseteq \Gamma_j\cap \Gamma_k$ from $v$ to $v_i$, for all triples
  $(i,j,k)$.
  Choose $v$ that minimizes the sum of length $|\gamma_1|+|\gamma_2|+|\gamma_3|$. We claim that each union $\gamma_i\cup\gamma_j$ is immersed.
  Indeed, suppose that some $\gamma_i\cup\gamma_j$ is not immersed. Then there exists $v'\neq v$ such that  $v'\in\gamma_i\cap\gamma_j$ and $\gamma_i\cap\gamma_j=P$, where $P$ is a path from $v'$ to $v$. But $\gamma_i\subseteq \Gamma_j\cap \Gamma_k$, for all triples $(i,j,k)$, so $P$ and in particular $v'$ belongs to $\Gamma_1\cap \Gamma_2 \cap \Gamma_3$. 
  Take $\gamma_i'=\gamma_i\setminus P$, $\gamma_j'=\gamma_j\setminus P$ and $\gamma_k'$ be an immersed subpath of $\gamma_k\cup P$. 
  It is clear that $$|\gamma'_1|+|\gamma'_2|+|\gamma'_3|\leq |\gamma_1|+|\gamma_2|+|\gamma_3|-|P|.$$ 
  By the minimality of the choice of $v$ we have $|P|=0$. Since $|P|$ is the length of path between $v$ and $v'$, it follows that $v'=v$, a contradiction.
  
  Since $\Gamma$ is path-cycle extensible and each $\gamma_i\cup\gamma_j$ is immersed, we may find a cycle in each $\Gamma_i$, for $i=1,2,3$, and a disc diagram $D$ consisting of cells $F_i$ mapped to those cycles, as in Figure~\ref{f:c4t4D}. Since complex $X_t$ satisfies $T(4)$ condition, this disc diagram cannot be reduced. Therefore, $D$ is reducible along a subpath of $\gamma_i$ for some $i$, which implies that $\Gamma_j$ and $\Gamma_k$ have the same image in $\Theta$. Thus $\Gamma_j$ contains $v_j$, a contradiction.
\end{proof}
\begin{figure}[h]
  \centering
  \includegraphics[width=0.24\textwidth]{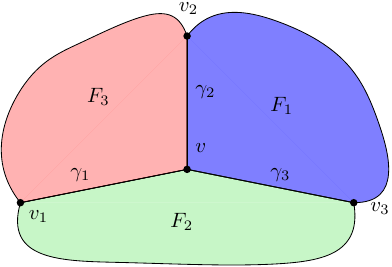}\caption{The proof of Lemma~\ref{l:c4t4strongH}.}\label{f:c4t4D}
\end{figure}

The property below follows directly from Lemmas \ref{l:c4t4strongH} and \ref{lem:intersections}(3). 
\begin{lm}[Helly property]\label{lem:flag}
Let $X$ be a simply connected \cftfs graphical complex.
Consider a collection $\{\Gamma_i \to \Theta\}_{i \in I}$ of graphical cells. If for every $i,j \in I$ the intersection $\Gamma_i \cap \Gamma_j$ is non-empty then the intersection $\bigcap_{i \in I} \Gamma_i$ is a non-empty tree.
\end{lm}

\section{Quadric complexes}\label{sec:quad}

In this section, we first recall the definition of quadric complexes, then we describe the quadrization of a non-thickened graphical complex, finally, we show that the quadrization of a simply connected non-thickened graphical \cftfs small cancellation complex is quadric.

\begin{df}\cite[Definition 1.1.]{HodaQuadric}
A \textit{locally quadric complex} is a square complex $Y$ satisfying the following conditions.
\begin{enumerate}[(A)]
\item The attaching map of every square is an immersion.
\item If there are two squares $F_1,F_2\in Y$ that they share at least three edges, then $\partial F_1=\partial F_2$.
\item If there are two squares $F_1,F_2\in Y$ such that $\partial(F_1\cup F_2)$ is a cycle of length $4$, then there is $F\in Y$ such that $\partial F=\partial(F_1\cup F_2)$.
\item If there are three squares $F_1,F_2,F_3\in Y$ such that $\partial(F_1\cup F_2\cup F_3)$ is a cycle of length $6$, then there exist $F,F'\in Y$ such that $\partial(F_1\cup F_2\cup F_3)=\partial(F\cup F')$, i.e.~this cycle has a diagonal that divides it into two $4$-cycles.
\end{enumerate}

Conditions (C) and (D) are called \textit{rules of replacement}. 
A \textit{quadric complex} is a simply connected locally quadric complex.
\end{df}

\begin{figure}[h]
\begin{center}
\includegraphics[scale=2.2]{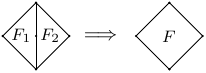}
\end{center}
\begin{center}
(C)
\end{center}
\begin{center}
\includegraphics[scale=2.2]{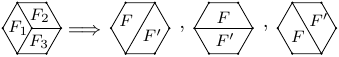}
\end{center}
\begin{center}
(D)
\end{center}
\caption{Rules of replacement for quadric complexes}
\end{figure}

\begin{pr}\label{mgchar}
\cite[Proposition 1.19]{HodaQuadric}
Let $X$ be a square complex.  The following are
  equivalent.
  \begin{enumerate}
  \item  $X$ is quadric.
  \item  $X$ is $4$-flag and the $1$-skeleton $X^1$
    of $X$ is $4$-bridged.
  \item $X$ is $4$-flag and simply connected and the
    $1$-skeleton $X^1$ of $X$ is simplicial and every $6$-cycle
    in $X^1$ has a diagonal.
  \end{enumerate}
\end{pr}

Let $X$ be a graphical complex with embedded graphical cells. 
Let $\mathcal{G}_X$ be the graph obtained from the $1$-skeleton of the non-thickened graphical complex $X_c$ by removing all the edges from $\Theta$.
The \textbf{quadrization} $Y_X$ of $X$ is the square complex obtained by gluing a unique square to every embedded $4$-cycle in the graph $\mathcal{G}_X$ (also called the \textbf{$4$-flag completion} of $\mathcal{G}_X$). 
Note that $\mathcal{G}_X$ is bipartite with one set of vertices corresponding to vertices of $\Theta$ (denoted $v_i$) and the other to tips of cones (denoted $\widehat{\Gamma}_i$).

If every $1$-cell of $X$ appears on the boundary of a graphical cell then the quadrization $Y_X$ of $X$ is connected.
For the rest of this section, we will assume that every $1$-cell of a graphical complex $X$ is contained in the boundary of at least one graphical cell.
This is not a serious restriction, since for each $1$-cell $e$ of $X$ not
appearing on the boundary of a graphical cell we
can subdivide it (if it is not
embedded) and consider it as $\Gamma'_e$. Take $\Gamma':=\Gamma\cup \Gamma'_e$. 
Thickened and non-thickened realizations of obtained graphical complex
$X'$ deformation retracts to $X$. The original complex $X$
embeds $Aut(X)$-equivariantly into $X'$ via a continuous map $X \to
X'$, which thus faithfully preserves group actions.  Furthermore, if
$X$ is \cftfs then so is $X'$.

The following lemmas are the graphical analogues of \cite[Lemma 3.9 and 3.10]{HodaQuadric}.

\begin{lm}\label{quadsc} Let $X$ be a graphical complex with embedded graphical cells. If $X$ is simply connected then so is its quadrization $Y_X$.
\end{lm}

\begin{proof}
  The proof follows the same strategy as the proof of \cite[Lemma 3.9]{HodaQuadric}, with the necessary modifications for graphical cells. 
  
  Suppose $Y_X$ is not simply connected.  
  Let $\alpha$ be a closed essential path in $Y_X$.
  Represent $\alpha$ by a cyclic sequence $\widehat{\Gamma}_0$, $v_1$, $\widehat{\Gamma}_2$, $v_3$, \ldots, $\widehat{\Gamma}_{2n-2}$, $v_{2n-1}$, $\widehat{\Gamma}_0$ of cone tips and vertices of $\Theta$, such that, for every $i$,
  the graphical cell $\Gamma_{2i}$ contains the $0$-cells $v_{2i-1}$ and $v_{2i+1}$ (indices modulo $2n$).  
  Let $\phi^j_{2i}$ be an immersed path in $\Gamma_{2i}$ from
  $v_{2i-1}$ to $v_{2i+1}$. 
  Let $\Phi_{2i}$ be the set of all such paths $\phi^j_{2i}$. 
  Let $\delta_0$, $\delta_2$, \ldots, $\delta_{2n-2}$ be any sequence with each $\delta_{2i} \in \Phi_{2i}$.  
  Then the concatenation $\delta = \delta_0 \delta_2 \cdots \delta_{2n-2}$ is a closed path in $\Theta \subset X_c$. 
  Since $X_c$ is simply connected, there is a disc diagram $D$ in $X_c$ with boundary path $\partial D = \delta$.
  
  Choose $\alpha$, $\delta_{2i}$ and $D$ so as to minimize
  the total number of $1$-cells and $2$-cells of $D$.
  Because every graphical cell $\Gamma_{2i}$ is embedded, any tip of a spur of $D$ must be
  one of the $v_{2i+1}$.  
  Let $u$ be the $0$-cell of the spur incident to this $v_{2i+1}$. Then $\widehat{\Gamma}_{2i}$, $v_{2i+1}$, $\widehat{\Gamma}_{2i+2}$,
  $u$, $\widehat{\Gamma}_{2i}$ is a $4$-cycle in $Y_X$ and so is nullhomotopic.  
  But then we can replace $v_{2i+1}$ with $u$ in $\alpha$ and replace $D$ with $D \setminus e$, where $e$ is the $1$-cell of the spur joining $u$ and $v_{2i+1}$.  
  This contradicts the minimality of our choices and therefore, we can conclude that $D$ has no spurs.
  
  Suppose $D$ has a $2$-cell. Then there is a $2$-cell $F$ with $\partial F\cap\partial D\neq\emptyset$. We remind the reader that all $2$-cells in $X_c$ are triangles with one vertex being the tip of a cone.
  But $\widehat{\Gamma}_i\not\in \delta$ for any $i$. Consequently, $\widehat{\Gamma}_F$ is internal vertex of $D$, thus there is a cycle in $D$ which corresponds to a cycle in $\Gamma_F$. 
  Let $\mathfrak F$ be the closed star in $D$ of the cone vertex $\widehat{\Gamma}_F$, that is, the union of all triangles of $D$ containing $\widehat{\Gamma}_F$.
  Let $w$ be a $0$-cell in the
  intersection of $\mathfrak{F}$ and $\partial D$, then $w$ is contained in some $\delta_{2i}$.  
  But then we can expand the vertex $\widehat{\Gamma}_{2i}$ in $\alpha$ to a subpath $\widehat{\Gamma}_{2i}$, $w$, $\widehat{\Gamma}_F$,
  $w$, $\widehat{\Gamma}_{2i}$ and replace $D$ by $D\setminus \mathfrak{F}$. The new disc diagram $D\setminus \mathfrak{F}$ contains at most one more $0$-cell and at least one less $2$-cell compared to $D$, contradicting the minimality of our choices.
  Thus, $D$ has no $2$-cells.
  
  Therefore, $D$ is a finite spurless tree, i.e., a single $0$-cell, $x$.  It
  follows that $v_{2i+1} = x$, for every $i$, so that $\alpha$ is
  $\widehat{\Gamma}_0$, $x$, $\widehat{\Gamma}_2$, $x$, \ldots, $\widehat{\Gamma}_{2n-2}$, $x$, $\widehat{\Gamma}_0$.  This path
  is clearly nullhomotopic, which is a contradiction.
\end{proof}

\begin{lm}
    Let $X$ be a simply connected graphical complex. If $X$ is
  \cftfs then its quadrization $Y_X$ is quadric.
\end{lm}

\begin{proof}
  By construction, $Y_X$ is $4$-flag. By Lemma~\ref{quadsc}, it is simply connected. Hence, by Proposition~\ref{mgchar}, it remains to show that every embedded $6$-cycle in $Y_X^1$ has a diagonal.
  An embedded $6$-cycle in $Y_X$ corresponds to a triple of pairwise intersecting graphical cells $\Gamma_1$, $\Gamma_2$ and $\Gamma_3$ in $X$ and three $0$-cells, one contained in each of the three pairwise intersections. 
  By Lemma \ref{l:c4t4strongH}, $\Gamma_1 \cap \Gamma_2 \subset \Gamma_3$, after possibly reindexing. Then the
  $0$-cell contained in $\Gamma_1 \cap \Gamma_2$ is incident to $\Gamma_3$. Hence the $6$-cycle has a diagonal.
\end{proof}

\section{Simply connected $C(6)$ graphical complexes}\label{sec:scv_graphical}

In this section, we recall and establish several properties of simply connected \css graphical complexes that will be used later. Many of these properties are due to Osajda and Prytu{\l}a \cite{OsajdaPrytula}. In a few places, we isolate intermediate steps from their arguments and formulate them separately, in order to make the later applications more convenient. The new results in this section concern intersection assembly trees, defined below.

The following Lemma is \cite[Lemma 6.10]{OsajdaPrytula}. It is the \css analogue of Lemma \ref{lem:intersections}.

\begin{lm}\label{lem:intersectionsc6}
Let $X$ be a simply connected \css graphical complex associated to a graphical presentation $\langle f :\Gamma \rightarrow \Theta \rangle$.
Then the following hold:
  \begin{enumerate}
  \item  For every graphical cell $\Gamma_i$, the map
    $\Gamma_i \to \Theta$ is an embedding.
  \item The intersection of (the images of) any two
    graphical cells is either empty or it is a finite tree.
  \item If three graphical cells pairwise intersect then
    they triply intersect and the intersection is a finite tree.		
  \end{enumerate}
\end{lm}

Unlike \cftfs graphical complexes, \css graphical complexes do not satisfy the strong Helly property in general; for instance, one may consider the regular tiling of the plane by hexagons. The proof of the Helly property in the $C(6)$ case therefore uses a different argument, due to \cite{OsajdaPrytula}. 
We shall need the following intermediate step from their proof.

\begin{pr}\label{relatortrees}
Let $X$ be a simply connected \css graphical complex, and let
$\{\Gamma_i \to \Theta\}_{i \in I}$ be a collection of pairwise intersecting graphical cells.
Then, for every $i\in I$, the subgraph
\[
T_i:=\Gamma_i\cap \bigcup_{j\in I\setminus\{i\}}\Gamma_j
=\bigcup_{j\in I\setminus\{i\}}(\Gamma_i\cap \Gamma_j)
\]
is a finite tree.
\end{pr}

We call $T_i$ the \textbf{intersection assembly tree} in $\Gamma_i$ associated to the collection $\{\Gamma_i \to \Theta\}_{i \in I}$.

\begin{rem}\label{piecesinrelatortrees}
    For any two vertices $u,v$ of the intersection assembly tree $T$, there exists $j,k\in I\setminus\{i\}$ such that $\Gamma_i\cap(\Gamma_{j}\cup \Gamma_{k})$ is a tree containing $u,v$. In particular, any immersed path in $T$ can be covered by at most two pieces.
\end{rem}

The following Helly property is \cite[Lemma 6.11]{OsajdaPrytula}. We include the short proof using Proposition \ref{relatortrees}, in order to fix the terminology of intersection assembly trees. 

\begin{lm}[Helly property for \cs]\label{lem:hellyc6}
Let  $X$ be a simply connected \css graphical complex associated to a graphical presentation $\langle f :\Gamma \rightarrow \Theta \rangle$.
Consider a collection $\{\Gamma_i \to \Theta\}_{i \in I}$ of graphical cells. If for every $i,j \in I$ the intersection $\Gamma_i \cap \Gamma_j$ is non-empty then the intersection $\bigcap\limits_{i \in I} \Gamma_i$ is a non-empty tree.
\end{lm}

\begin{proof}

Let $0\in I$, and let $T := \Gamma_0\cap \bigcup_{i \in I\setminus \{0\}} \Gamma_i$ be the intersection assembly tree in $\Gamma_0$ associated to the collection $\{\Gamma_i \to \Theta\}_{i \in I}$.
By Lemma~\ref{lem:intersectionsc6}, all intersections $\{\Gamma_0 \cap \Gamma_i\}_{i \in I\setminus \{0\}}$ are pairwise intersecting, non-empty subtrees of $T$. 
Since $T$ is a finite tree, the Helly property for subtrees applies to arbitrary collections of subtrees of $T$. Hence, $\bigcap_{i\in I}\Gamma_i=\bigcap_{i\in I\setminus\{0\}}(\Gamma_0\cap \Gamma_i)$ is a non-empty tree.
\end{proof}

\begin{lm}\label{lem:relatortreesintersections}
Let $X$ be a simply connected \css graphical complex associated to a graphical presentation $\langle f :\Gamma \rightarrow \Theta \rangle$ and $\Gamma_i$ be a graphical cell.
Then the intersection of any two intersection assembly trees in $\Gamma_i$ is either empty or a finite tree.
Moreover, if $\Gamma_i$ is not essentially $C(6)$ then any triple of intersection assembly trees in $\Gamma_i$ that pairwise intersects, triply intersect and the intersection is a finite tree.		
\end{lm}

\begin{proof}
Let $T,T'$ be two intersection assembly trees in $\Gamma_i$ with non-empty intersection. Since both of them are finite trees, the intersection $T\cap T'$ is a forest. 
Assume that $T\cap T'$ contains two disjoint connected components $S,S'$. 
Choose any $u\in S$, $v\in S'$. There are paths $P\to T$, $P'\to T'$  joining $u$ with $v$. Since $S$ and $S'$ are disjoint connected components of $T\cap T'$. 
Without loss of generality, we can assume that the concatenation $PP'$ is an immersed cycle in $\Gamma_i$.
By Remark \ref{piecesinrelatortrees} paths in $T$ (resp. $T'$) between any point of $S$ and any point of $S'$ can be covered by at most two pieces. 
Consequently both $P$ and $P'$ can be covered by at most two pieces each, and cycle $PP'$ can be covered by at most four pieces, a contradiction to condition \cs. 
Therefore, $T\cap T'$ has a single connected component and is a subtree of a finite tree.

Let $T_1,T_2,T_3$ be three intersection assembly trees in $\Gamma_i$ with pairwise non-empty intersection. 
As discussed in the previous paragraph, each pairwise intersection is a finite tree. Assume that the triple intersection $T_1\cap T_2\cap T_3$ is empty. Take vertices $v_1,v_2,v_3$ such that $v_i\notin T_i$ and $v_1\in T_2\cap T_3$, $v_2\in T_1\cap T_3$ and $v_3\in T_1\cap T_2$.
There are paths $P_1\to T_1$ joining $v_2$ with $v_3$, $P_2\to T_2$ joining $v_1$ with $v_3$ and $P_3\to T_3$ joining $v_1$ with $v_2$. Without loss of generality, we can assume that concatenation $P_1P_2P_3$ is an immersed cycle in $\Gamma_i$.
Again, these paths can be covered by at most two pieces each, and cycle $P_1P_2P_3$ can be covered by at most six pieces, a contradiction to the fact that $\Gamma_i$ is not essentially  \cs. Therefore, the triple intersection is non-empty. To show that it is a finite tree, repeat the reasoning from the first part of the proof to the intersections of $T_1\cap T_2$ and $T_2\cap T_3$.

\end{proof}

\begin{lm}[Helly property for intersection assembly trees]\label{lem:hellyrelatortrees} Let $X$ be a simply connected \css graphical complex associated to a graphical presentation  $\langle f :\Gamma \rightarrow \Theta \rangle$  and $\Gamma_0$ be a graphical cell which is not essentially \cs.
Consider a collection of intersection assembly trees $T_i, i\in I$ in $\Gamma_0$. 
If for every $k,l$ the intersection $T_k \cap T_l$ is non-empty then the intersection $\bigcap\limits_{i\in I} T_i$ is a non-empty tree.
\end{lm}

\begin{proof}
Choose $j\in I$. Since $T_j$ is finite tree, it has finitely many subtrees and without loss of generality we can assume that $I$ is finite.
By Lemma~\ref{lem:relatortreesintersections} all intersections $\{T_j \cap T_i\}_{i\in I\setminus\{j\}}$ are pairwise intersecting, non-empty subtrees of a tree $T_j$.
Therefore, by the Helly property for trees $\bigcap\limits_{i\in I} T_i$ it is a non-empty tree.
\end{proof}

We remind the reader that we are especially interested in \textbf{torsion-essentially \cs-free} graphical small cancellation complexes. 
Such a complex $X$ is a \css graphical small cancellation complex where every finite order automorphism with no fixed point on the $1$-skeleton of $X$ is not essentially $C(6)$.

The torsion-essentially \css-free condition is automatic, for example, if every essentially \css graphical cell has the property that every automorphism of its $1$-skeleton fixes a point. 
For example it can be a $\Theta$ graph (i.e. a graph obtained by subdivision of edges of graph consisting of three edges joining a single pair of vertices), with uneven sides. 
Another type of examples of such complexes are complexes where some automorphisms of $1$-skeleton of essentially \css cell are fixed point free, but do not extend to an automorphism of whole complex. Basic example is septagon glued to 6 other cells, 5 of which intersect along single edge and one intersecting along two edges.

\section{Systolic complexes}\label{sec:sys}

In this section, we first recall the definition of systolic complexes, then we describe the Wise complex of a non-thickened graphical complex. Finally, we show that the Wise complex of a simply connected non-thickened graphical \css small cancellation complex is systolic.

The study of groups acting on systolic complexes was introduced by T. Januszkiewicz and J. Świątkowski \cite{JŚ06} and independently by Haglund \cite{syshag}.
Systolic complexes might be seen as a simplicial version of the cubical combinatorial nonpositive curvature theory of $\mathrm{CAT(0)}$ cube complexes. We follow here the notation of \cite{JŚ06}.

\begin{df}
Let $X$ be a simplicial complex. 
$X$ is $k$-\textit{large} if $X$ is flag and every cycle in $X$ of length less than $k$ has a diagonal, i.e. in $X$ there is an edge connecting nonconsecutive vertices of the cycle.
\end{df}

\begin{df}
A simplicial complex $X$ is \textit{systolic} if $X$ is simply connected, and the links of all vertices of $X$ are $6$-large.
\end{df}

\begin{df}\label{def:unionofconecells} Let $X$ be a simply connected \css graphical complex associated to a graphical presentation $\langle f :\Gamma \rightarrow \Theta \rangle$.
Assume that $X$ is the union of its graphical cells,  i.e.\ $\Theta = f(\Gamma)$.
Let ${\bf U}$ be the covering of $X$ by these graphical cells. Define the simplicial complex $W(X)$ to be the nerve of the covering ${\bf U}$. 
This complex was introduced by D.\ Wise in the classical $C(p)$ setting \cite{Wise}, therefore we will refer to $W(X)$ as the \emph{Wise complex}. 
\end{df}

\begin{thm}\label{t:ss}
\cite[Theorem 7.10]{OsajdaPrytula}
Let $X$ be a simply connected \css graphical complex satisfying the conditions of definition \ref{def:unionofconecells}. Then its Wise complex $W(X)$ is systolic.
\end{thm}

\section{Distributed Euler characteristic and curvature}\label{sec:cv}

In this section, we recall the definition of the Euler characteristic, define the distributed Euler characteristic and the curvature of a vertex, and prove several properties of curvature for square disc diagrams.

Let $X$ be a finite $2$-complex.
The \textbf{Euler characteristic} $\chi(X)$ is the sum of number of $0$- and $2$-cells of $X$, reduced by the number of $1$-cells of $X$.
Note that every disc diagram has Euler characteristic equal to $1$ and every spherical diagram has Euler characteristic equal to $2$.

The \textbf{distributed Euler characteristic} of a vertex $v\in X$ is obtained by taking $1$ (corresponding to Euler characteristic of $v$), then adding $\frac{1}{n}$ for each $n$-gon containing $v$, and subtracting $\frac{1}{2}$ for each edge incident to $v$.
The terminology comes from distributing the contribution $-1$ of each $1$-cell equally between its two endpoints, and the contribution $1$ from each $2$-cell equally among its boundary $0$-cells.
Therefore, summing up the distributed Euler characteristic over all vertices of $X$ yields the Euler characteristic of $X$. 

\subsection{Square disc diagrams}
Let $D$ be a disc diagram that is a square complex. Throughout this subsection, such diagrams will be called \textbf{square disc diagrams}.

The \textbf{square curvature} of a vertex $v\in D$ is defined as
$$
\kappa^{\square}_D(v)=4 - 2\delta_D(v) + \rho^{\square}_D(v),
$$
where $\delta_D(v)$ denotes the valence of $v$ and $\rho^{\square}_D(v)$ denotes the number of squares incident to $v$, cf. Fig.\ref{fig:class}. Note that, for a square disc diagram, the curvature of a vertex is exactly four times its distributed Euler characteristic.

\begin{figure}[h!]
\begin{center}
\includegraphics[scale=1.5]{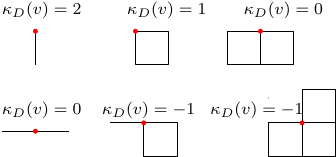}
\end{center}
\caption{All possible neighborhoods of a boundary vertex $v$ (red) of a square disc diagram $D$ such that $-1\leq\kappa^{\square}_D(v)\leq 2$.}\label{fig:class}
\end{figure}

We now state a version of the combinatorial Gauss-Bonnet theorem for CAT(0) square disc diagrams first introduced by Hoda in \cite[Proposition 1.8]{HodaQuadric}.

\begin{pr}[Gauss-Bonnet Theorem for CAT(0) Square Disc Diagrams]\label{gb}
Let $D$ be a CAT(0) square disc diagram. Then: 
$$
\sum\limits_{v \in D}\kappa^{\square}_D(v)=4 
$$
moreover:
$$
\sum\limits_{v \in \partial D}\kappa^{\square}_D(v)\geq 4
$$
\end{pr}

The following proposition states a well known property (see. e.g. \cite[Proposition 6.2]{dudaall3}) of curvature along geodesics.

\begin{pr}\label{geo}
Let $D$ be a square disc diagram and let $\gamma\subset\partial D$ be a geodesic in $D$. Then no internal vertex of $\gamma$ has curvature greater than $1$. Moreover, if $u,v$ are internal vertices of $\gamma$ with $\kappa_D(u)=\kappa_D(v)=1$, then there is a vertex $w\in \gamma$ between $v$ and $u$ with $\kappa_D(w)\leq -1$.
\end{pr}

\begin{df}
    Let $D$ be a square disc diagram and $\gamma:=\{x'_0,x_0,\ldots x_n,x'_n\}$ a path in its boundary.
    Suppose that $\kappa_D(x_i)$ is equal to $0$ for $1\leq i\leq n-1$ and equal to $1$ for $i\in\{1,n\}$.
    Then there exists a path 
    $$\gamma':=\{x'_0\ldots x'_n\},$$ 
    and a subcomplex $\mathcal{P}$ of $D$ bounded by $\gamma\cup\gamma'$ consisting of squares $$F_i:=\{x_i,x_{i+1},x'_{i+1},x'_i\},$$ cf. Fig.\ref{fig:class}.
    The subcomplex $\mathcal{P}$ is called a \textbf{plateau along} $\gamma$.
\end{df}

\begin{figure}[h!]
\begin{center}
\includegraphics[scale=1]{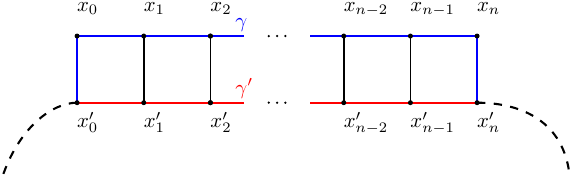}
\end{center}
\caption{Example of a plateau}\label{fig:plateau}
\end{figure}

From now on, by the \textbf{curvature along} $\gamma$ we mean the sum of the curvature of all internal vertices of $\gamma$.

\begin{pr}\label{plateau}
    Let $D$ be a square disc diagram and $\gamma$ an injective path in its boundary. Assume that every internal vertex of $\gamma$ has curvature at most $1$, and that the sum of the curvature along $\gamma$ is at least $2$. Then there exists a subpath $\beta\subset\gamma$ together with a plateau along $\beta$.
\end{pr}

\begin{proof}
    Since each internal vertex of $\gamma$ has curvature at most $1$ and the sum is at least $2$, there exists a subpath $\beta$ with curvature along $\beta$ equal to $2$ and any proper subpath of $\beta$ of curvature strictly smaller than $2$.
    By definition, there is a plateau along $\beta$.
\end{proof}

The following Lemma relates quadric complexes and CAT(0) square complexes.

\begin{lm}\cite[Lemma 1.6]{HodaQuadric}\label{quad}
Let $Y$ be a quadric complex and $\alpha$ be some cycle in $Y$. If $D$ is a minimal area disc diagram for $\alpha$ in $Y$, then $D$ is a CAT(0) square complex.
\end{lm}

\subsection{Triangle disc diagrams.}

We define curvature analogously for triangle complexes. Let $D$ be a disc diagram that is a triangle complex. Throughout this subsection, such diagrams will be called \textbf{triangle disc diagrams}.

The \textbf{triangle curvature} of a vertex $v\in D$ is defined as follows
$$
\kappa^{\triangle}_D(v)=6 - 3\delta_D(v) + 2\rho^{\triangle}_D(v),
$$
where $\rho^{\triangle}_D(v)$ denotes the number of triangles incident to $v$, cf. Fig.\ref{fig:class2}. Note that, for a triangle disc diagram, the curvature of a vertex is exactly six times its distributed Euler characteristic.

\begin{figure}[h!]
\begin{center}
\includegraphics[scale=1.5]{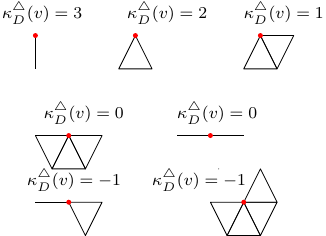}
\end{center}
\caption{All possible neighborhoods of a boundary vertex $v$ (red) of a triangle disc diagram $D$ such that $-1\leq\kappa^{\triangle}_D(v)\leq 3$.}\label{fig:class2}
\end{figure}

We now state properties of triangle curvature, analogous to properties from the previous subsection. 
We begin with a version of the combinatorial Gauss-Bonnet theorem for CAT(0) triangle disc diagrams.

\begin{pr}[Gauss-Bonnet Theorem for CAT(0) Triangle Disc Diagrams]\label{gbsys}
Let $D$ be a CAT(0) triangle disc diagram. Then: 
$$
\sum\limits_{v \in D}\kappa^{\triangle}_D(v)=6, 
$$
moreover:
$$
\sum\limits_{v \in \partial D}\kappa^{\triangle}_D(v)\geq 6. 
$$
\end{pr}

\begin{pr}\label{geosys}
Let $D$ be a triangle disc diagram and let $\gamma\subset\partial D$ be a geodesic in $D$. Then no internal vertex of $\gamma$ has curvature greater than $1$. Moreover, if $u,v$ are internal vertices of $\gamma$ with $\kappa_D(u)=\kappa_D(v)=1$, then there is a vertex $w\in \gamma$ between $v$ and $u$ with $\kappa_D(w)\leq -1$.
\end{pr}

\begin{df}
    Let $D$ be a triangle disc diagram and $\gamma:=\{x'_0,x_0,\ldots x_n,x'_{n+1}\}$ a path in its boundary.
    Suppose that $\kappa_D(x_i)$ is equal to $0$ for $1\leq i\leq n-1$ and equal to $1$ for $i\in\{1,n\}$.
    Then there exists a path $$\gamma':=\{x'_0\ldots x'_{n+1}\},$$ and a subcomplex $\mathcal{P}$ of $D$ bounded by $\gamma\cup\gamma'$ consisting of triangles $$\hat{F}_i:=\{x'_{i},x_{i},x'_{i+1}\} \text{ for } 0\leq i \leq n \text{ and } \check{F}_i:=\{x_{i},x_{i+1},x'_{i+1}\} \text{ for } 0\leq i\leq n-1$$ cf. Fig.\ref{fig:class2}.
    The subcomplex $\mathcal{P}$ is called a \textbf{plateau along} $\gamma$.
\end{df}

\begin{figure}[h!]
\begin{center}
\includegraphics[scale=1]{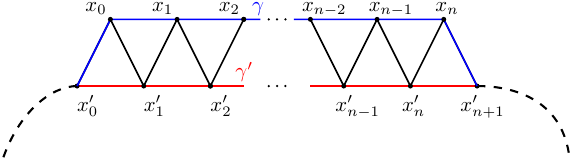}
\end{center}
\caption{Example of the plateau in a triangle disc diagram}\label{fig:plateausys}
\end{figure}

\begin{pr}\label{plateautriangle}
    Let $D$ be a triangle disc diagram and $\gamma$ an injective path in its boundary. Assume that the curvature of each internal vertex of $\gamma$ is at most $1$ and the sum of the curvature along $\gamma$ is at least $2$. Then there exists a subpath $\beta\subset\gamma$ and a plateau along $\beta$.
\end{pr}

\begin{proof}
    Since each internal vertex of $\gamma$ has curvature at most $1$ and the sum is at least $2$, there exists a subpath $\beta$ with curvature along $\beta$ equal to $2$ and any proper subpath of $\beta$ of curvature strictly less than $2$.
    By definition, there is a plateau along $\beta$.
\end{proof}

The following lemma relates systolic complexes and CAT(0) triangle complexes.
This follows from \cite[Claim 1, Theorem 8.1]{CHEPOI2000125}. We note that the \textbf{bridged complexes} considered there are precisely systolic complexes.

\begin{lm}\label{sys}
Let $X$ be a systolic complex and $\alpha$ be some cycle in $X$. If $D$ is a minimal area disc diagram for $\alpha$ in $X$, then $D$ is a CAT(0) triangle complex.
\end{lm}

\section{Finite order automorphism fixes a cell}\label{sec:fix}

For many simply connected non-positively curved complexes, an automorphism fixes a point if and only if it (or more precisely, the cyclic subgroup generated by it) has bounded orbits. However, this fails for non-locally finite $C(p)$--$T(q)$ small cancellation complexes with $q<5$. 

To see this, consider the following complex $X$ (see Figure \ref{fig:infclique}). The $1$-skeleton $X^1$ is the infinite clique spanned by $\mathbb{Z}$. We attach a $2$-cell along every cycle of the form $v_i,v_{i+1},\ldots v_{i+k}, v_i$ for any $i$ and $k\geq 2$. 

\begin{figure}[h]
    \centering
    \includegraphics[width=0.7\linewidth]{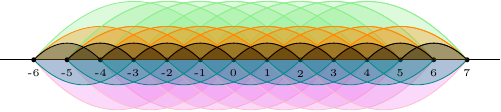}
    \caption{A non-locally finite $C(p)$--$T(4)$ complex}
    \label{fig:infclique}
\end{figure}
Since none of the $2$-cells is fully covered by pieces, $X$ is $C(p)$ for any $p$. Moreover, $X$ satisfies $T(q)$ for $q\leq 4$.
Note that there is an obvious action of $\mathbb{Z}$ on $X$. This action has a bounded orbit, as $X$ is a bounded set of diameter $1$. However, the action of $\mathbb{Z}$ on $X$ has no fixed point.

Since we do not assume local finiteness for graphical \cftfs or \css complexes, an action of a group whose cyclic subgroups all have bounded orbits need not have a fixed point. 
Nevertheless, we show that every action of a finite cyclic group on a simply connected graphical \cftfs or \css complex admits a fixed point. 

We begin by proving that simply connected non-thickened \cftfs and \css graphical complexes are contractible. 
Related results were obtained by Osajda and Prytu{\l}a, who proved that  locally finite non-thickened graphical $C(6)$ complexes are contractible \cite{OsajdaPrytula}. 
Moreover, Gruber proved that the presentation complex associated to a graphical presentation satisfying a stronger form of the $C(6)$ condition is aspherical \cite{Gruber-graphical}. 
Our result does not require local finiteness, and in the $C(6)$ case it applies under the usual graphical $C(6)$ condition.

Let $X$ be a graphical complex with associated presentation $\langle f:\Gamma\rightarrow \Theta \rangle$. We remind the reader that we assume that all graphical complexes are simplified, this assumption is necessary for Lemma \ref{lm:fixpt}, as otherwise a finite cyclic group acting on $X_c$ would not necessarily stabilize a cone cell of $X_c$, but instead stabilize the union of multiple cone cells, all spanned by the same $1$-skeleton.

Recall that a simply connected 2-complex $Y$ is contractible if and only if its second homotopy group $\pi_2(Y)$ is trivial. 
Indeed, by the Hurewicz theorem, $\pi_2(Y)\cong H_2(Y)$, and since the complex is 2-dimensional, the vanishing of $\pi_2$ implies the vanishing of all higher homotopy groups; the conclusion then follows from Whitehead's theorem.

Therefore, if $X$ is a simply connected graphical complex and its non-thickened realization $X_c$ is not contractible, then $\pi_2(X_c)$ is nontrivial. 
Hence, there exists a spherical diagram $S\rightarrow X_c$ representing a nontrivial element of $\pi_2(X_c)$. 
We call such a spherical diagram \textbf{minimal} if it has the least number of 2-cells among all homotopically nontrivial spherical diagrams in $X_c$. 

\begin{pr}\label{spheresthicknonthick}
Let $X_c$ be a non-thickened graphical complex with $\pi_2(X_c)$ nontrivial.  
Then, for every minimal homotopically nontrivial spherical diagram $S\rightarrow X_c$, there is an associated spherical diagram $S_t\rightarrow X_t$ in the thickened realization.    
\end{pr}

\begin{proof}
    Every $2$-cell of $S$ is a triangle. 
    Moreover, for every triangle $T$ in $S$, there is some $Cone(\Gamma_i)$ such that one vertex of $T$ maps to the cone tip $\widehat{\Gamma}_i$.
    Let $v$ be a vertex of $S$ mapping to a cone tip $\widehat{\Gamma}_i$. Denote by $St_S(v)$ the closed star of $v$ in $S$, that is, the union of all closed cells of $S$ containing $v$. 
    Since $S$ is a minimal spherical diagram, the link $S_v$ is a single cycle. 
    Moreover, $S_v$ corresponds to an immersed cycle in $\Gamma_i$.
    Indeed, if the corresponding cycle in $\Gamma_i$ is not immersed, then one could reduce the spherical diagram near $v$ and obtain a homotopically nontrivial spherical diagram with fewer 2-cells, contradicting the minimality of $S$.  
    Let $C\rightarrow \Gamma_i$ be the immersed cycle corresponding to the link $S_v$, and let $F_v$ be the 2-cell of the thickened graphical complex $X_t$ with the attaching map $C\rightarrow \Gamma_i\rightarrow \Theta$. 
    We now replace $St_S(v)$ by the cell $F_v$, for every vertex $v$ of $S$ mapping to a cone tip. After performing this replacement at all such vertices, we obtain a spherical diagram $S_t\rightarrow X_t$. 
\end{proof}

\begin{lm}
    Let $X$ be a simply connected \cftfs or \css graphical small cancellation complex, then its non-thickened realization $X_c$ is contractible.
\end{lm}

\begin{proof}
    Suppose, for a contradiction, that $X_c$ is not contractible.
    Since $X_c$ is a simply connected $2$-complex, the discussion preceding Proposition \ref{spheresthicknonthick} implies that $\pi_2(X_c)$ is nontrivial. 
    Hence, there exists a homotopically nontrivial spherical diagram in $X_c$. 
    Let $S\rightarrow X_c$ be a minimal one.
    
    By Proposition \ref{spheresthicknonthick}, there is an associated spherical diagram $S_t\rightarrow X_t$ in the thickened graphical complex.
    We first claim that $S_t\rightarrow X_t$ cannot be reduced. 
    Indeed, if $S_t\rightarrow X_t$ were reduced, then Proposition \ref{reduceddiagram} would imply that $S_t$ satisfies the \cftfs or \css condition. 
    We show that this is impossible for a spherical diagram. 
    
    We may assume that $S_t$ has no vertices of degree $2$: whenever such a vertex occurs, we remove it together with its incident edges and replace them by a single edge. 
    This operation does not decrease the relevant $C(p)$ or $T(q)$ parameters.
    Since $S_t$ is spherical, every vertex of $S_t$ is internal. 
    
    In the \cftfs case, the $T(4)$ condition implies that every vertex $v$ of $S_t$ has degree $n\geq 4$. If $n$ cells contain $v$, then the distributed Euler characteristic at $v$ satisfies 
    $\chi_{S_t}(v)\leq 1-\frac{n}{2}+\frac{n}{4}\leq 0$.
    
    In the \css case, every vertex has degree $n\geq 3$, and the $C(6)$ condition implies that each incident 2-cell contributes at most $\frac{1}{6}$ to the distributed Euler characteristic at $v$. 
    Thus $\chi_{S_t}(v)\leq 1-\frac{n}{2}+\frac{n}{6}\leq 0$.
    
    In both cases, summing over all vertices gives $\chi_{S_t}\leq 0$ contradicting $\chi_{S_t}= 2$. 
    Therefore $S_t\rightarrow X_t$ is not reduced.
    
    Thus $S_t\rightarrow X_t$ is reducible along some piece.
    By Proposition \ref{reducible_spherical}, either there exists a disc subdiagram $D_t\subseteq S_t$ such that $D_t\rightarrow X_t$ is reducible, or $S_t$ consists of two $2$-cells with the same boundary and mapping to the same thick cell of $X_t$.
    
    Consider first the case where such a reducible disc subdiagram $D_t$ exists. 
    Then there is a piece $P\rightarrow D_t$ along which $D_t\rightarrow X_t$ is reducible. Let $F_1$, $F_2$ be the $2$-cells in $D_t$ sharing the piece $P\rightarrow D_t$, then the boundaries of $F_1$ and $F_2$ must map to the same graphical cell.  
    Hence, a reduction can be performed in $S$: the cone pieces corresponding to $F_1$ and $F_2$ have the same cone tip, and after identifying the cone tips and canceling the duplicate triangles based on the piece $P\rightarrow D_t$, one obtains a spherical diagram $S'\rightarrow X_c$ with fewer $2$-cells than $S$ and the same homotopy class as $S\rightarrow X_c$, contradicting the minimality of $S$.
    
    It remains to consider the second case, where $S_t$ consists of two $2$-cells with the same boundary and mapping to the same thick cell of $X_t$. Then the associated diagram $S\rightarrow X_c$ is contained in a single cone of $X_c$. Hence it is homotopically trivial, contradicting the choice of $S$.

    Therefore, no homotopically nontrivial spherical diagram exists in $X_c$, so $\pi_2(X_c)=0$. Since $X_c$ is simply connected, it follows that $X_c$ is contractible.

    \end{proof}
    
    We will use the following fixed-point theorem of Casacuberta-Dicks: every action of a finite solvable group on a 2-dimensional $\mathbb{Z}$-acyclic CW complex has a fixed point \cite[Corollary 1.2]{Casacuberta1992}. Since cyclic groups are solvable, and since a contractible complex is $\mathbb{Z}$-acyclic, we obtain the following consequence. 

    \begin{cor}
    \label{cor:fixpt}
      Let $X$ be a simply connected \cftfs or \css graphical small cancellation complex, and let $X_c$ be the associated non-thickened graphical complex.
      Then every action of finite cyclic group on $X_c$ has a fixed point.
    \end{cor}

    \begin{lm}\label{lm:fixpt}
        Let $X$ be a simply connected \cftfs or \css graphical small cancellation complex. Let $G$ be a finite cyclic group acting on $X_t$, such that the action of $G$ on the $1$-skeleton of $X_t$ is free.
    Then $G$ stabilizes a thick cell of $X_t$.
    \end{lm}

    \begin{proof}
        Consider the associated non-thickened graphical complex $X_c$. 
        By Corollary \ref{cor:fixpt}, the action of $G$ on $X_c$ has a fixed point. By the freeness of the action on $X_t$ and therefore on $\Theta$, it fixes a point inside of a cone. Clearly it has to fix a tip of a cone. If tip of a cone is fixed, then thick cell corresponding to a cone is stabilized.
    \end{proof}

\section{Concatenations of geodesics}\label{sec:main}

Let $X$ be a simply connected \cftfs or \css graphical small cancellation complex such that every $1$-cell of a
non-thickened graphical complex $X_c$ is contained in the boundary of at least one cone.
Let $G$ be a group acting on $X$ by automorphisms and assume that this action induces a free action on the $1$-skeleton $X^{1}$ of $X$. 

Let $Y$ be either the quadrization or the Wise complex of $X$ (depending whether $X$ is \cftfs or \cs). 
For any pair of vertices $x,y\in Y$ corresponding to two graphical cells of $X$, by $\mathfrak{G}(x,y)$ we denote the set of all geodesics between $x$ and $y$ and by $\mathfrak{X}_i(x,y)$, be the set of all vertices at the distance $i$ from $x$ in $\mathfrak{G}(x,y)$.
\subsection{Case of $C(4)$--$T(4)$}
\begin{lm}\label{lmklcftf}
Let $f\in G$ be a nontrivial element of finite order $m$.
Let $x\in \mathrm{Fix}_Y(f)$ and $y$ be a vertex distinct from $x$.
There exists $k$, such that $\mathfrak{X}_1(x,y)\cap f^k\mathfrak{X}_1(x,y)=\emptyset$.
\end{lm}

\begin{proof}
We claim that there exists $\mathfrak{x}\in \mathfrak{X}_2(x,y)$ such that $\mathfrak{X}_1(x,y)$ is the set of all vertices from $\mathfrak{x}\cap x$.
Then there exists $f^k$ such that $\mathfrak{X}_1(x,y)\cap f^k\mathfrak{X}_1(x,y)=\emptyset$. Indeed, if there is no such $k$ then $x,\mathfrak{x}, f\mathfrak{x},\ldots ,f^{m-1}\mathfrak{x}$ is a set of pairwise intersecting graphical cells. Thus, by Lemma \ref{lem:flag}, there exists a subtree stabilized by $f$ contained in the $x\cap\mathfrak{x}\cap f\mathfrak{x}\cap \ldots \cap f^{m-1}\mathfrak{x}$. A contradiction with the fact that $f$ acts freely on the $1$-skeleton of $X$.

We will now prove the claim.
Observe that any pair of geodesics between $x$ and $y$ spans a minimal area disc diagram $D$. If $x_i, x_j\in \mathfrak{X}_1(x,y)$ are distinct, then $D$ has nonzero area. But curvature along geodesics is at most $1$, implying that curvature of $x$ is also $1$ (it cannot be $2$, as $x_i$ and $x_j$ are distinct).
Consequently there is exactly one square, containing $x,x_i,x_j$ and some vertex $\mathfrak{x}_{i,j}$.
Assume that $\mathfrak{x}_{i,j}\not\in \mathfrak{X}_2(x,y)$. Then distance from $y$ to $\mathfrak{x}_{i,j}$ is equal to distance from $x$ to $y$, and there are two geodesics one going through $x_i$ and one through $x_j$. A minimal area disc diagram $D'$ bounded by them can be obtained from $D$ by removal of a square $[x,x_i,\mathfrak{x}_{i,j},x_j]$. By the same argument as above, there exists a square containing $\mathfrak{x}_{i,j},x_j,x_j$ and some vertex $\mathfrak{x}'_{i,j}$ in $D'$. Since $D'$ is obtained from $D$ by removal of a single square, this square is also in $D$. But then squares $[x,x_i,\mathfrak{x}_{i,j},x_j]$ and $[\mathfrak{x}_{i,j},x_i,\mathfrak{x}'_{i,j},x_j]$ have common two edges in $D$, and therefore can be replaced by a single square, contradicting the minimality of the area of $D$.

We will show that there is $\mathfrak{x}\in \mathfrak{X}_2(x,y)$, such that $\mathfrak{X}_1(x,y)$ is contained in $x\cap \mathfrak{x}$. 
We show by induction. By the discussion above, for every pair of vertices in $\mathfrak{X}_1(x,y)$, there is a cell in $\mathfrak{X}_2(x,y)$ containing both vertices. 
Assume that every collection of $k$ vertices in $\mathfrak{X}_1(x,y)$ belong to some graphical cell in $\mathfrak{X_2}(x,y)$. 
Take $k+1$ vertices $x_1, ..., x_{k+1}\in \mathfrak{X}_1(x,y)$. 
By the induction hypothesis, there is $\mathfrak{x}_i\in \mathfrak{X}_2(x,y)$ containing all but $x_i$. 
Since $\mathfrak{x}_1, \mathfrak{x}_2, \mathfrak{x}_3$ pairwise intersect, by strong Helly property (Lemma \ref{l:c4t4strongH}), one of these relators contains the intersection of the remaining two. 
Without loss of generality, assume $\mathfrak{x}_1\cap\mathfrak{x}_2 \subset \mathfrak{x}_3$. As $x_3\in \mathfrak{x}_1\cap \mathfrak{x}_2$, we have $x_3\in \mathfrak{x_3}$, therefore, $\mathfrak{x}_3$ contains all $k+1$ vertices. 
Since $\mathfrak{X}_1(x,y)$ is finite, after some step $\mathfrak{X}_1(x,y)$ is contained in $x\cap \mathfrak{x}$.

\end{proof}

\begin{rem}
    Note that the case of graphical small cancellation complexes deviates here significantly from the case of classical small cancellation complexes \cite{dudaall3}.
    While in the classical small cancellation setting, 
    $$\exists k\forall y, \mathfrak{X}_1(x,y)\cap f^k\mathfrak{X}_1(x,y)=\emptyset,$$
    in the graphical setting, there is no such global power and we can only obtain weaker property 
    $$\forall y\exists k, \mathfrak{X}_1(x,y)\cap f^k\mathfrak{X}_1(x,y)=\emptyset.$$ 
    Figure \ref{fig:petersenc4t4} shows an example of a fixed graphical cell where such global power of group element rotating it cannot exist.

    \begin{figure}[h!]
        \centering
        \includegraphics[width=0.4\linewidth]{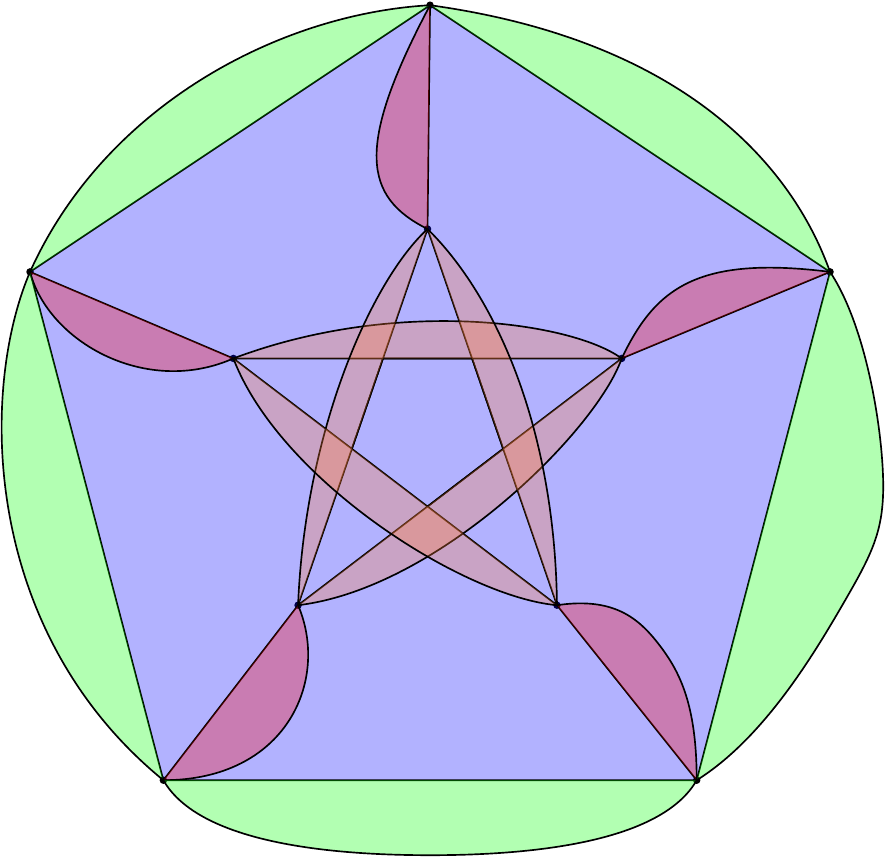}
        \caption{The group element $g$ of order $5$ acts on the central graphical cell (blue), whose $1$-skeleton is the Petersen graph, by rotation. The vertices $y_1,y_2\in Y$ correspond to one of green and one of orange cells, respectively. For every non-trivial $g^k$ if $\mathfrak{X}_1(x,y_1)\cap g^k\mathfrak{X}_1(x,y_1)=\emptyset$, then $\mathfrak{X}_1(x,y_2)\cap g^k\mathfrak{X}_1(x,y_2)\neq\emptyset$}
        \label{fig:petersenc4t4}
    \end{figure}
\end{rem}

\begin{lm}\label{lem:doubling}

Let $x,y,z$ be three distinct vertices in $Y$, if $\mathfrak{X}_1(x,y)\cap \mathfrak{X}_1(x,z)=\emptyset$, then the concatenation of any geodesic from $y$ to $x$ and from $x$ to $z$ is again a geodesic.

\end{lm}

\begin{proof}

Let $\gamma$ be a geodesic from $x$ to $y$ and let $\eta$ be a geodesic from $x$ to $z$. 

First observe that $\gamma\cap \eta=\{x\}$, that is, the concatenation of the inverse of $\gamma$ and $\eta$, denoted as $\gamma \cup \eta$, is an injective path. 
To see that, let $\gamma:=\{x=x_0, x_1,\ldots, x_m=y\}$, and let $\eta := \{x = x'_0, \ldots, x'_n = z\}$. Suppose there is $x_i = x'_j$ in $\gamma\cap \eta \supsetneq \{x\}$, then $i = d(x,x_i) = d(x,x'_j) = j$, and $\{x = x'_0, \ldots, x'_i = x_i, x_{i+1}, \ldots, x_m = y\}$ will be a geodesic from $x$ to $y$. If $i>0$, then $x'_1\in \mathfrak{X}_1(x,y)\cap \mathfrak{X}_1(x,z)$, contradicting our assumption, therefore, $i$ has to be 0, i.e., $x$ is the only common vertex of $\gamma$ and $\eta$.   

Let $\alpha$ be a geodesic in $Y$ between $y$ and $z$, and let $D_\alpha$ be a disc diagram in $Y$ bounded by $\gamma \cup \eta$ and $\alpha$. Choose $\alpha$ and $D_\alpha$ in a way that minimizes the area of $D_\alpha$.

Suppose that $\gamma\cup \eta$ is not a geodesic, then $D_\alpha$ contains at least one square.
By the minimality of $D_\alpha$, the intersection of $\gamma$ (resp. $\eta$) and $\alpha$ in $D_\alpha$ is a subpath of $\gamma$ (resp. $\eta$), between $y$ (resp. $z$) and some $y'$ (resp. $z'$). 
Therefore, $D_\alpha$ consists of a nonsingular disc diagram and two attached (potentially trivial) paths, where the disc is bounded by two paths from $y'$ to $z'$, one via $\gamma \cup \eta$ and the other via $\alpha$, and attached paths are $\gamma\cap \alpha$ and $\eta \cap \alpha$, attached at $y'$ and $z'$, respectively, cf. Figure \ref{fig:doublingdisc}. 
\begin{figure}[h!]
\begin{center}
\includegraphics[scale=1]{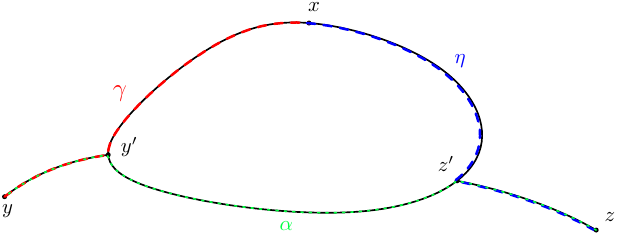}
\end{center}
\caption{Disc diagram $D_\alpha$}\label{fig:doublingdisc}
\end{figure}
Since $D_\alpha$ is a minimal area disc diagram in $Y$, by Lemma \ref{quad}, $D_\alpha$ is CAT(0), and then by Proposition \ref{gb}, the sum of curvature over vertices in its boundary $\partial D_\alpha$ is at least $4$. Let $\beta$ be the subpath of $\gamma\cup\eta$ between $y'$ and $z'$. Then the curvature along $\beta$ and the curvature of every vertex in $\alpha$ should sum up to be at least 4.

By Proposition \ref{geo}, each internal vertex of $\alpha$ has curvature at most $1$. If any vertex has curvature exactly $1$, then it belongs to a boundary square $S$ with two edges along $\alpha$. Thus, $\alpha'$ obtained from $\alpha$ by switching to the internal edges of $S$ is also a geodesic. But $D_{\alpha'}$ has at least one square fewer than $D_{\alpha}$, a contradiction with the choice of $\alpha$.
Therefore, every internal vertex of $\alpha$ has nonpositive curvature in $D_\alpha$, and the ends $y$ and $z$ can have curvature at most 2.
We argue that the sum of curvature over vertices in $\alpha$ is at most 2. 
If it were not, then the curvature of at least one of $y,z$ is 2. 
Assume $\kappa_D(y)=2$, then $y\neq y'$, and $y'$ has curvature at most $-1$, i.e. $\kappa_D(y')\leq -1$. Therefore, the curvature of $z$ also needs to be 2, which similarly implies that $\kappa_D(z')\leq -1$. Since $y'\neq z'$, the sum of curvature over vertices in $\alpha$ cannot exceed 2. In particular, this implies that the curvature along $\beta$ is at least 2.

Since $\beta$ is an injective boundary path in $D_\alpha$, every internal vertex of $\beta$ has curvature at most 1, therefore, by Proposition \ref{plateau}, there exists a plateau $\mathcal{P}$ along $\beta$. 
Let $p_l,p_r$ denote the vertices of positive curvature at the ends of $\mathcal{P}$.
Since all vertices between $p_l$ and $p_r$ have curvature $0$, by Proposition \ref{geo} vertices $p_l,p_r$ cannot both be internal vertices of a single geodesic. Without loss of generality assume that $p_l\in \gamma$ and $p_r\in \eta$. Note that it is possible that one of $p_l$, $p_r$ is $x$.
Let $x'$ be a vertex at the bottom of $\mathcal{P}$ that is adjacent to $x$. By the existence of $\mathcal{P}$, there are geodesics $\gamma'$ from $x$ to $y$ through $x'$ and $\eta'$ from $x$ to $z$ through $x'$, cf. Figure \ref{fig:plateaudoubling}. Consequently, $x'\in \mathfrak{X}_1(x,y)\cap \mathfrak{X}_1(x,z)$. However, this contradicts the condition that the intersection is empty. Hence, $D_\alpha$ contains no square, and $\gamma\cup \eta$ is a geodesic. 

\begin{figure}[h!]
\begin{center}
\includegraphics[scale=1]{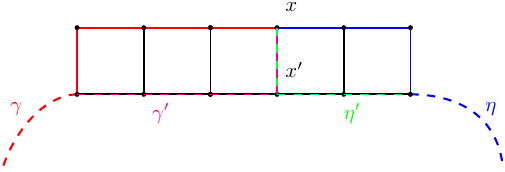}
\end{center}
\caption{Example of the plateau $\mathcal{P}$ and paths $\gamma',\eta'$}\label{fig:plateaudoubling}
\end{figure}

\end{proof}

\subsection{Case of \css with stabilized cell that is not essentially \css }

\begin{lm}\label{lmklcsv}
Let $f\in G$ be a nontrivial element of finite order $m$. 
Let $x\in \mathrm{Fix}_Y(f)$ be a vertex of Wise complex corresponding to a not essentially \css cell of complex $X$. Let $y$ be a vertex distinct from $x$.
There exists $k$, such that $N(\mathfrak{X}_1(x,y))\cap f^k\mathfrak{X}_1(x,y)=\emptyset$, where $N(\mathfrak{X}_1(x,y))$ is the set of all vertices of $Y$ at distance at most $1$ from any vertex of $\mathfrak{X}_1(x,y) $.
\end{lm}

\begin{proof}
 
For two vertices $x_i, x_j\in \mathfrak{X}_1(x,y)$, we can choose two geodesics in $Y$ between $x$ and $y$ passing through $x_i$ and $x_j$ respectively, and we can fill these two geodesics with a disc diagram in $Y$. 
Among all disc diagrams obtained in this way, let $D\rightarrow Y$ be one that minimizes the area. 
If $x_i, x_j\in \mathfrak{X}_1(x,y)$ are distinct, then $D$ has nonzero area, the minimality of $D$ then implies that $D$ is a nonsingular disc diagram attached with a (possibly trivial) path $\alpha$. 
Note that the further end of $\alpha$ is precisely the vertex $y$. 
Denote by $y'$ the vertex where $\alpha$ intersects the nonsingular subdiagram.
Note that the sum of curvature of the vertices in $\alpha$ does not exceed $2$ - if $\alpha$ is trivial, then the curvature of $y$ is at most $2$, if $\alpha$ is nontrivial, then the curvature of $y$ is $3$, while the curvature of $y'$ is negative.
Since the curvature along a geodesic is at most $1$, the curvature of $x$ in $D$ must be at least $2$. 
Since $x_i$ and $x_j$ are distinct, the curvature of $x$ cannot be $3$, therefore, it has to be exactly $2$. 
Consequently, $x_i,x_j$ must be joined by an edge.

Hence, the whole set $\mathfrak{X}_1(x,y)$ spans a clique in $Y$.
The intersection of the graphical cell corresponding to $x$ with the union of the graphical cells corresponding to the vertices in $\mathfrak{X}_1(x,y)$ is clearly an intersection assembly tree $T$ in $x$.
If $N(\mathfrak{X}_1(x,y))\cap f^k\mathfrak{X}_1(x,y)\neq\emptyset$, then there exist $x_i,x_j\in \mathfrak{X}_1(x,y)$ such that the graphical cells corresponding to $x_i$ and $f^kx_j$ have nonempty intersection. 
Consequently, by the Helly property for \css complexes, the triple intersection of the graphical cells corresponding to $x,x_i$ and $f^kx_j$ is non-empty, in particular, $T\cap f^k T\neq\emptyset$. 
It is also easy to see that $T\cap f^k T\neq\emptyset$ implies $N(\mathfrak{X}_1(x,y))\cap f^k\mathfrak{X}_1(x,y)\neq\emptyset$.
Therefore, the condition $N(\mathfrak{X}_1(x,y))\cap f^k\mathfrak{X}_1(x,y)=\emptyset$ is equivalent to $T\cap f^k T=\emptyset$. 

Assume that $T\cap f^k T\neq\emptyset$ for all $k\in \{1,...,m-1\}$. 
Since $x$ represents a not essentially $C(6)$ cell, by Lemma \ref{lem:hellyrelatortrees}, $T\cap fT\cap f^2T\cap \ldots \cap f^{m-1}T$ is a nonempty tree stabilized by $f$, therefore, there must be a fixed point, contradicting the fact that $f$ acts freely on the $1$-skeleton.
\end{proof}

\begin{rem}
    As in the case of \cftfs graphical small cancellation, there is significant deviation from the case of classical small cancellation complexes \cite{dudaall3}.
    In the classical case, $$\exists k\forall y, N(\mathfrak{X}_1(x,y))\cap f^k\mathfrak{X}_1(x,y)=\emptyset$$ unless $f^3=id$. In the graphical \css case not only there is no such global power, even weaker property $$\forall y\exists k,N(\mathfrak{X}_1(x,y))\cap f^k\mathfrak{X}_1(x,y)=\emptyset$$  can be obtained only for a cell that is not essentially \cs.
    Figure \ref{fig:petersenc6} shows an example of essentially \css fixed relator where such power cannot exist for a group element of order $5$. The analogues of such thickened cells exist for rotations of arbitrary order. Note that in case of even orders, such action by the rotation is not free on $1$-skeleton. 

    \begin{figure}[h!]
        \centering
        \includegraphics[width=0.4\linewidth]{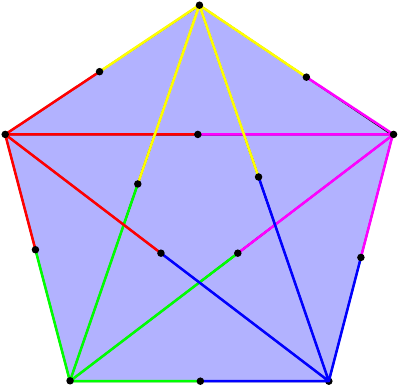}
        \caption{The group element $g$ of order $5$ acts on the central graphical cell (blue), whose $1$-skeleton is $5$-clique with subdivided edges, by rotation. The intersection assembly tree corresponding to $\mathfrak{X}_1(x,y)$ is marked in red. This tree has intersection with its images (other colors) for every power of $g$. Consequently for every non-trivial $g^k$ the intersection $N(\mathfrak{X}_1(x,y))\cap g^k\mathfrak{X}_1(x,y)$ is non-empty.}
        \label{fig:petersenc6}
    \end{figure}
\end{rem}

\begin{lm}\label{lem:doublingsys}

Let $x,y,z$ be three distinct vertices in $Y$, if $N(\mathfrak{X}_1(x,y))\cap \mathfrak{X}_1(x,z)=\emptyset$, then the concatenation of any geodesic from $y$ to $x$ and from $x$ to $z$ is again a geodesic.
\end{lm}

\begin{proof}

Let $\gamma$ be a geodesic from $x$ to $y$ and let $\eta$ be a geodesic from $x$ to $z$. 

First observe that $\gamma\cap \eta=\{x\}$, that is, the concatenation of the inverse of $\gamma$ and $\eta$, denoted as $\gamma \cup \eta$, is an injective path. 
To see that, let $\gamma:=\{x=x_0, x_1,\ldots, x_m=y\}$, and let $\eta := \{x = x'_0, \ldots, x'_n = z\}$. Suppose there is $x_i = x'_j$ in $\gamma\cap \eta \supsetneq \{x\}$, then $i = d(x,x_i) = d(x,x'_j) = j$, and $\{x = x'_0, \ldots, x'_i = x_i, x_{i+1}, \ldots, x_m = y\}$ will be a geodesic from $x$ to $y$. If $i>0$, then $x'_1\in \mathfrak{X}_1(x,y)\cap \mathfrak{X}_1(x,z)\subset N(\mathfrak{X}_1(x,y))\cap \mathfrak{X}_1(x,z)$, contradicting our assumption, therefore, $i$ has to be 0, i.e., $x$ is the only common vertex of $\gamma$ and $\eta$.   

Let $\alpha$ be a geodesic in $Y$ between $y$ and $z$, and let $D_\alpha$ be a disc diagram in $Y$ bounded by $\gamma \cup \eta$ and $\alpha$. Choose $\alpha$ and $D_\alpha$ in a way that minimizes the area of $D_\alpha$.

Suppose that $\gamma\cup \eta$ is not a geodesic, then $D_\alpha$ contains at least one triangle.
As in the proof of Lemma \ref{lem:doubling}, $D_\alpha$ consists of a nonsingular disc diagram and two attached (potentially trivial) paths, where the disc is bounded by two paths from $y'$ to $z'$, one via $\gamma \cup \eta$ and the other via $\alpha$, and attached paths are $\gamma\cap \alpha$ and $\eta \cap \alpha$, attached at $y'$ and $z'$, respectively, cf. Figure \ref{fig:doublingdisc}.

Since $D_\alpha$ is a minimal area disc diagram in $Y$, by Lemma \ref{sys}, $D_\alpha$ is CAT(0), and then by Proposition \ref{gbsys}, the sum of curvature over vertices in its boundary $\partial D_\alpha$ is at least $6$. Let $\beta$ be the subpath of $\gamma\cup\eta$ between $y'$ and $z'$. Then the curvature along $\beta$ and the curvature of every vertex in $\alpha$ should sum up to be at least 6.

By Proposition \ref{geosys}, each internal vertex of $\alpha$ has curvature at most $1$. If any vertex has curvature exactly $1$, then it belongs to a boundary square $S$ with two edges along $\alpha$. Thus, $\alpha'$ obtained from $\alpha$ by switching to the internal edges of $S$ is also a geodesic. But $D_{\alpha'}$ has at least two fewer triangles than $D_{\alpha}$, a contradiction with the choice of $\alpha$.
Therefore, every internal vertex of $\alpha$ has nonpositive curvature in $D_\alpha$, and the ends $y$ and $z$ can have curvature at most 3.
We argue that the sum of curvature over the vertices in $\alpha$ is at most 4. 
If it were not, then the curvature of at least one of $y,z$ is 3. 
Assume $\kappa_D(y)=3$, then $y\neq y'$, and $y'$ has curvature at most $-1$, i.e. $\kappa_D(y')\leq -1$. Similarly, if the curvature of $z$ is 3, then $\kappa_D(z')\leq -1$. 
Since $y' \neq z'$, the sum of curvature over vertices in $\alpha$ cannot exceed 4.
In particular, the curvature along $\beta$ is at least 2.

Observe that every internal vertex of $\beta$ has curvature at most 1. 
Indeed, the curvature of $x$ cannot be $2$, as $N(\mathfrak{X}_1(x,y))\cap \mathfrak{X}_1(x,z)=\emptyset$, and every other internal vertex of $\beta$ is an internal vertex of either the geodesic $\gamma$ or $\eta$ and therefore has curvature at most $1$ by Proposition \ref{geosys}. 
Hence, by Proposition \ref{plateautriangle}, there exists a plateau $\mathcal{P}$ along $\beta$. 
Let $p_l,p_r$ denote the vertices of positive curvature at the ends of $\mathcal{P}$.
Since all vertices between $p_l$ and $p_r$ have curvature $0$, by Proposition \ref{geosys} vertices $p_l,p_r$ cannot both be internal vertices of a single geodesic. Without loss of generality, assume that $p_l\in \gamma$ and $p_r\in \eta$. Note that it is possible that one of $p_l$, $p_r$ is $x$.
Let $x'_l,x'_r$ be vertices at the bottom of $\mathcal{P}$ that are adjacent to $x$. By the existence of $\mathcal{P}$, there are geodesics $\gamma'$ from $x$ to $y$ through $x'_l$ and $\eta'$ from $x$ to $z$ through $x'_r$, cf. Figure \ref{fig:plateaudoublingsys}. Consequently, $x'_r\in N(\mathfrak{X}_1(x,y))\cap \mathfrak{X}_1(x,z)$. However, this contradicts the condition that the intersection is empty. Hence, $D_\alpha$ contains no triangle, and $\gamma\cup \eta$ is a geodesic. 

\begin{figure}[h!]
\begin{center}
\includegraphics[scale=1]{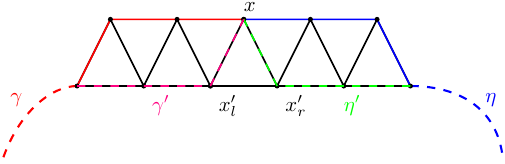}
\end{center}
\caption{Example of the plateau $\mathcal{P}$ and paths $\gamma',\eta'$}\label{fig:plateaudoublingsys}
\end{figure}

\end{proof}

\section{Lack of global fixed point implies existence of group element of infinite order}\label{sec:main2}

Throughout this section, we retain the assumptions on the complexes $X,Y$ and the group $G$ from the previous section.

First, we observe the following:

\begin{pr}\label{pr:singlefix}
For every non trivial $f\in G$ of finite order, $\mathrm{Fix}_Y(f)$ is nonempty. Moreover, in the \cftfs case, $\mathrm{Fix}_Y(f)$ consists of a unique vertex representing a graphical cell in $X$; in the \css case, if $\mathrm{Fix}_Y(f)$ contains a vertex representing a not essentially $C(6)$ cell, then $\mathrm{Fix}_Y(f)$ contains no other vertex. 
\end{pr}

\begin{proof}
In the case of a \css complex, all vertices of $Y$ correspond to graphical cells of $X$. In the case of \cftf, there are two sets of vertices in $Y$, one corresponding to graphical cells of $X$ and one corresponding to vertices in $X$. 
By our assumption of freeness, there can be no vertices in $\mathrm{Fix}_Y(f)$ that represent vertices of $X$.

By Lemma \ref{lm:fixpt}, every action of a cyclic subgroup of $G$ fixes the tip of some cone in a non-thickened graphical complex $X_c$, and therefore stabilizes a graphical cell in $X$. Thus $\mathrm{Fix}_Y(f)$ contains at least one such vertex of $Y$. 

If an element $f\in G$ fixes two distinct graphical cells $\mathfrak{x}$ and $\mathfrak{y}$, let $\gamma$ be a geodesic between the vertices $x$ and $y$ corresponding to these cells in $Y$.
In the \cftfs case, by Lemmas \ref{lmklcftf} and \ref{lem:doubling}, there exists a power of $f$ such that $\gamma\cup f^k\gamma$ is a geodesic between $x$ and $f^kx =x$, a contradiction. 
In the \css case, assume that $\mathfrak{x}$ is not essentially $C(6)$, then by Lemma \ref{lmklcsv} and \ref{lem:doublingsys}, a similar contradiction follows.
\end{proof}

From now on, we assume that $x\in \mathrm{Fix}_Y(f)$ is a vertex corresponding to a graphical cell $\mathfrak{x}$ of $X$ stabilized by $f$. 
In the \css case, we additionally assume that the graphical cell $\mathfrak{x}$ is not essentially \cs.

\begin{lm}\label{lem:inford}
If $g\in G$ is an element that does not fix $x$, then there exists $w\in \langle f,g\rangle$ with infinite order.
\end{lm}

\begin{proof}
Without loss of generality, we may assume that $g$ is a nontrivial element of finite order, and $\mathrm{Fix}_Y(g)$ contains exactly one vertex $y$ which is distinct from $x$ and represents a graphical cell in $X$. 
In the $C(6)$ case, we additionally assume that the graphical cell which $y$ represents is not essentially $C(6)$. 
If any of these assumptions is not satisfied by the original $g$, we can simply carry out the proof with $\bar{g}=gfg^{-1}$ and $\bar{y}=gx$ and establish the conclusion for $\bar{g}$, the conclusion for $g$ then follows.

We show the proof in the case of \cftfs complexes. The case of \css  complexes is analogous, with differences explained at the end of the proof.

Suppose that every element in $\langle f,g\rangle$ has finite order.
Take any geodesic $\gamma$ between $x$ and $y$. 
Set $v_0:=f,u_{0,0}:=g$, $y_{0,0}:=y$. To improve the readability of the proof, for now we will continue using the notation $f,g,y$.
By Lemmas \ref{lmklcftf} and \ref{lem:doubling} we can find $k_0$ such that $\gamma_{0,0}:=\gamma\cup f^{k_0 }\gamma$ is a geodesic.
Let $m$ be the order of $g$, and let the admissible exponent set $\mathfrak{P}(y,f^{k_0}y;g)$ (resp. $\mathfrak{P}(f^{k_0}y,y;f^{k_0}gf^{-k_0})$) be the set of all possible $0< \ell < m$ such that $\mathfrak{X}_1(y,f^{k_0}y)\cap g^\ell\mathfrak{X}_1(y,f^{k_0}y)=\emptyset$, (resp. $\mathfrak{X}_{1}(f^{k_0}y,y)\cap f^{k_0}g^\ell f^{-k_0}\mathfrak{X}_{1}(f^{k_0}y,y)=\emptyset$). 
By definition both those sets have at most $m-1$ elements, and their nonemptiness follows from Lemma \ref{lmklcftf}.

Set $u_{0,1}:= v_{0}^{k_0}u_{0,0}v_{0}^{-k_0}$,  $y_{0,1}:=v_{0}^{k_0}y_{0,0}$, and $w_{0,0}$ as the trivial element. It is clear that $u_{0,1}$ fixes $y_{0,1}$, $u_{0,1}$ is again a nontrivial finite order element in $G$, and $y_{0,1}$ corresponds to a graphical cell in $X$.  
Therefore, by Lemma \ref{lmklcftf}, $\mathfrak{P}(y_{0,1},y_{0,0};u_{0,1}) = \mathfrak{P}(f^{k_0}y,y;f^{k_0}gf^{-k_0})$ is nonempty.

Fix an arbitrary element $l_0\in \mathfrak{P}(y_{0,1},y_{0,0};u_{0,1})$. We now define, inductively and for as long as possible, vertices $y_{0,n}$, group elements $u_{0,n}$ and $w_{0,n}$, and geodesics $\gamma_{0,n}$.
    
For $n\geq 1$, suppose that $y_{0,n-1}$, $y_{0,n}$, $u_{0,n-1}$, $u_{0,n}$, $\gamma_{0,n-1}$, $w_{0,n-1}$ have already been defined, and satisfy the following properties:
$y_{0,n-1}$ and $y_{0,n}$ are vertices in $Y$ representing graphical cells;
$u_{0,n-1}$ and $u_{0,n}$ are two nontrivial finite order elements in $G$ fixing $y_{0,n-1}$ and $y_{0,n}$ respectively; 
$\gamma_{0,n-1}$ is a geodesic between $y_{0,0}$ and $y_{0,n}$; 
$w_{0,n-1}$ is an element in $G$ such that $w_{0,n-1}\gamma_{0,0}$ is a subpath of $\gamma_{0,n-1}$ between $y_{0,n-1}$ and $y_{0,n}$.

If $l_0\notin \mathfrak{P}(y_{0,n},y_{0,0};u_{0,n})$, then the construction stops at step $n-1$. 
Otherwise, we define
$\gamma_{0,n}:=\gamma_{0,n-1}\cup u^{l_0}_{0,n} w_{0,n-1}\gamma_{0,0}$,
$u_{0,n+1}:=u^{l_0}_{0,n}u_{0,n-1}u^{-l_0}_{0,n}$, $y_{0,n+1}:=u^{l_0}_{0,n}y_{0,n-1}$, and 
$w_{0,n}:=u^{l_0}_{0,n} w_{0,n-1}$. 
Then $y_{0,n+1}$ again represents a graphical cell in $X$, and $u_{0,n+1}$ is again a nontrivial finite order element of $G$ fixing $y_{0,n+1}$.
Moreover, by Lemma \ref{lem:doubling}, $\gamma_{0,n}=\gamma_{0,n-1}\cup u^{l_0}_{0,n} w_{0,n-1}\gamma_{0,0}$ is a geodesic from $y_{0,0}$ to $y_{0,n+1}$. 
Finally, by the definition of $w_{0,n}$, the path $w_{0,n}\gamma_{0,0}$ is a subpath of $\gamma_{0,n}$ between $y_{0,n}$ and $y_{0,n+1}$. 

We claim that this construction must stop after finitely many steps. 
Suppose, for a contradiction, that it never stops.
Then the union $\gamma_{0,\infty}:=\bigcup_{n\geq 0}\gamma_{0,n}$ is an infinite geodesic ray. 
Moreover, the subpaths $w_{0,n}\gamma_{0,0}$ appear consecutively along $\gamma_{0,\infty}$. 
Since $u_{0,n+1}^{l_0}u_{0,n}^{l_0}=u_{0,n}^{l_0}u_{0,n-1}^{l_0}$, it follows inductively that
$$
w_{0,n}=\left\{\begin{array}{cc}
     (u_{0,1}^{l_0}u_{0,0}^{l_0})^{\frac{n}{2}} & \mbox{for even } n
     \\
     (u_{0,1}^{l_0}u_{0,0}^{l_0})^{\lfloor\frac{n}{2}\rfloor} u_{0,1}^{l_0}& \mbox{for odd } n.
\end{array}\right.
$$
In particular, $w_{0,2n}=(w_{0,2})^n$ for every $n$.
Hence $w_{0,2}$ acts on $\gamma_{0,\infty}$ by a non-trivial translation, and therefore has infinite order. 
This contradicts our assumption that $\langle f,g\rangle$ is a torsion group, since $w_{0,2}\in \langle f,g\rangle$. Therefore the construction must stop after finitely many steps.


We refer to the construction with fixed first index $j$ as the extension at the $j$-th level. 
Thus the construction above is the extension at the $0$-th level. 
Since it stops after finitely many steps, there is a first integer $n_0\geq 1$ such that
$y_{0,n_0},\quad u_{0,n_0},\quad \gamma_{0,n_0-1}$ are defined, but 
$l_0\notin \mathfrak P(y_{0,n_0},y_{0,0};u_{0,n_0})$.
Equivalently, the extension at the $0$-th level terminates after $n_0-1$ steps.

We now repeat the same construction at successive levels.
Suppose that, for some $j\geq 1$, the construction has already been carried out up to level $j-1$, and the extension at the $(j-1)$-th level terminates after $n_{j-1}-1$ steps. 

In particular, assume that $y_{j-1,0}$, $y_{j-1,n_{j-1}}$, $u_{j-1,0}$, $u_{j-1,n_{j-1}}$,$\gamma_{j-1,n_{j-1}-1}$ have already been defined such that $u_{j-1,0}$ and $u_{j-1,n_{j-1}}$ fix $y_{j-1,0}$ and $y_{j-1,n_{j-1}}$ respectively, and 
$\gamma_{j-1,n_{j-1}-1}$ is a geodesic between $y_{j-1,0}$ and $y_{j-1,n_{j-1}}$.


Set $v_{j}:=u_{j-1,0}$, $y_{j,0}:=y_{j-1,n_{j-1}}$, and $u_{j,0}:=u_{j-1,n_{j-1}}$.
Choose $k_j\in \mathfrak{P}(y_{j-1,0},y_{j,0};v_j)$ and define $y_{j,1}:=v_{j}^{k_j} y_{j,0}$,  $u_{j,1}:=v_j^{k_j}u_{j,0}v_j^{-k_j}$, and 
$\gamma_{j,0}:=\gamma_{j-1,n_{j-1}-1}\cup v_j^{k_j}\gamma_{j-1,n_{j-1}-1}.$
By Lemma \ref{lem:doubling}, $\gamma_{j,0}$ is a geodesic from $y_{j,0}$ to $y_{j,1}$. 
Now choose $l_j\in \mathfrak{P}(y_{j,1},y_{j,0};u_{j,1})$.

As before, we define $y_{j,i}$, $u_{j,i}$, $w_{j,i}$, and $\gamma_{j,i}$ inductively for as long as possible. 
Suppose $i\geq 1$, and assume that these objects have been defined up to step $i-1$ satisfying: $\gamma_{j,i-1}$ is a geodesic between $y_{j,0}$ and $y_{j,i}$, and $w_{j,i-1}\gamma_{j,0}$ is a subpath of $\gamma_{j,i-1}$ between $y_{j,i-1}$ and $y_{j,i}$.
If $l_j\notin \mathfrak{P}(y_{j,i},y_{j,0};u_{j,i})$, then the construction stops at step $i-1$. 
Otherwise, define
$w_{j,i}:=u_{j,i}^{l_j}w_{j,i-1}$,
$\gamma_{j,i}:=\gamma_{j,i-1}\cup w_{j,i}\gamma_{j,0}$, 
$u_{j,i+1}:=u_{j,i}^{l_j}u_{j,i-1}u_{j,i}^{-l_j}$,
$y_{j,i+1}:=u_{j,i}^{l_j}y_{j,i-1}$,
where $w_{j,0}$ is the trivial element, see Figure \ref{fig:extension}.
\begin{figure}
    \centering
    \includegraphics[width=\linewidth]{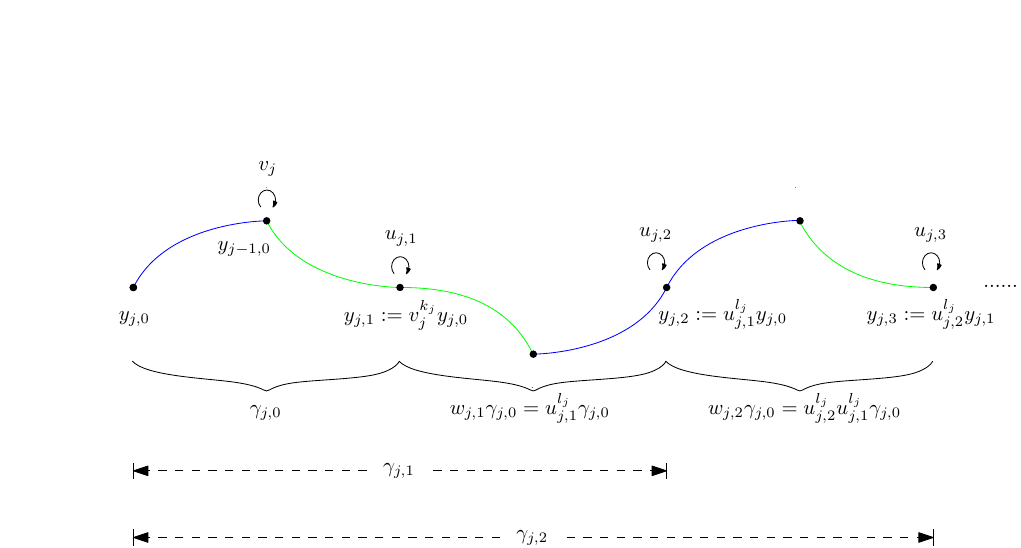}
    \caption{Extension at the $j$-th level}
    \label{fig:extension}
\end{figure}
Then $u_{j,i+1}$ fixes $y_{j,i+1}$, and by Lemma \ref{lem:doubling}, $\gamma_{j,i}$ is a geodesic from $y_{j,0}$ to $y_{j,i+1}$, and $w_{j,i}\gamma_{j,0}$ is its terminal subpath between $y_{j,i}$ and $y_{j,i+1}$.

Again, we claim that the extension at the $j$-th level must stop after finitely many steps. 
Indeed, suppose that it never stops, then the union $\gamma_{j,\infty}:=\bigcup_{i\geq 0}\gamma_{j,i}$ is an infinite geodesic ray. 
As in case $j=0$, the relations $u_{j,i+1}^{l_j}u_{j,i}^{l_j}
=u_{j,i}^{l_j}u_{j,i-1}^{l_j}$ give 
$$
w_{j,i} =\left\{\begin{array}{cc}
     (u_{j,1}^{l_{j}}u_{j,0}^{l_{j}})^{\frac{i}{2}} & \mbox{if } i \mbox{ is even} 
     \\
     (u_{j,1}^{l_{j}}u_{j,0}^{l_{j}})^{\lfloor\frac{i}{2}\rfloor} u_{j,1}^{l_{j}}& \mbox{if } i \mbox{ is odd}.
\end{array}\right.
$$
In particular, $w_{j,2i}=(w_{j,2})^i$ for every $i$.
Hence $w_{j,2}$ acts on $\gamma_{j,\infty}$ by a non-trivial translation, and therefore has infinite order. 
This contradicts the assumption that every element of $\langle f,g\rangle$ has finite order, since $w_{j,2}\in \langle f,g\rangle$. 
Therefore the construction must stop after finitely many steps. Let $n_j\geq 1$ be the first integer such that $y_{j,n_j},\quad u_{j,n_j},\quad \gamma_{j,n_j-1}$
are defined, but 
$l_j\notin \mathfrak P(y_{j,n_j},y_{j,0};u_{j,n_j})$.
Thus the extension at the $j$-th level terminates after $n_j-1$ steps.

We have therefore shown that, assuming $\langle f,g\rangle$ is a torsion group, the construction can always be continued to the next level, and therefore $j$ can be arbitrarily large. 
We now show that this is impossible. The obstruction is that the sets of admissible exponents form a strictly descending sequence of non-empty subsets of a fixed finite set.

\begin{figure}[h!]
\begin{center}
\includegraphics[scale=0.8]{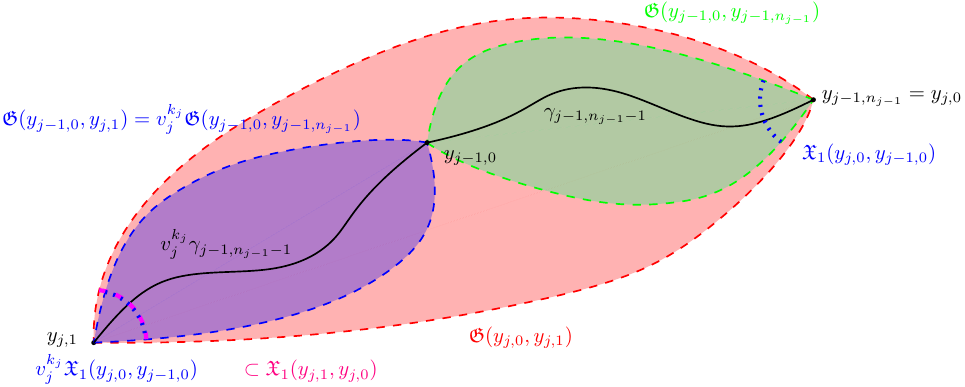}
\end{center}
\caption{Inclusions of geodesic sets used to compare admissible exponent sets.}
\label{fig:geodesicsets}
\end{figure}

\begin{cl}
Set $y_{-1,0}\coloneqq x$, for each $j\geq 0$, $\mathfrak{P}(y_{j+1,0},y_{j,0};u_{j+1,0}) \subsetneq \mathfrak{P}(y_{j,0},y_{j-1,0};u_{j,0})$.
\end{cl}

\begin{proof}
Recall that $\mathfrak{P}(y_{j+1,0},y_{j,0};u_{j+1,0}) = \mathfrak{P}(y_{j,n_j},y_{j,0};u_{j,n_j})$, it suffices to show $\mathfrak{P}(y_{j,n_j},y_{j,0};u_{j,n_j})$ $\subsetneq \mathfrak{P}(y_{j,0},y_{j-1,0};u_{j,0})$.
Observe that for all $i\geq 2$, we have
\[
\mathfrak{P}(y_{j,i},y_{j,i-1};u_{j,i})
= \mathfrak{P}(u_{j,i-1}^{-l_j} y_{j,i} ,y_{j,i-1}; u_{j,i-1}^{-l_{j}}u_{j,i}u_{j,i-1}^{l_j})
= \mathfrak{P}(y_{j,i-2} ,y_{j,i-1}; u_{j,i-2})
\]
and $\mathfrak{P}(y_{j,i-1},y_{j,i};u_{j,i-1}) = \mathfrak{P}(y_{j,i-1},u_{j,i-1}^{l_{j}} y_{j,i-2};u_{j,i-1})= \mathfrak{P}(y_{j,i-1} ,y_{j,i-2}; u_{j,i-1})$. Therefore,
for odd $i$, we have
$$\mathfrak{P}(y_{j,i},y_{j,i-1};u_{j,i})
= \mathfrak{P}(y_{j,1},y_{j,0};u_{j,1}).$$ 
As $y_{j-1,0}$ lies on a geodesic between $y_{j,0}$ and $y_{j,1}$, we have $ \mathfrak{X}_1(y_{j,1},y_{j-1,0})\subset \mathfrak{X}_1(y_{j,1},y_{j,0})$, cf. Figure \ref{fig:geodesicsets}, which implies 
$$\mathfrak{P}(y_{j,1},y_{j,0};u_{j,1})
\subset \mathfrak{P}(y_{j,1},y_{j-1,0};u_{j,1}).$$
Similarly, for even $i$, we have
$$\mathfrak{P}(y_{j,i},y_{j,i-1};u_{j,i}) 
= \mathfrak{P}(y_{j,0},y_{j,1};u_{j,0}) \subset \mathfrak{P}(y_{j,0},y_{j-1,0};u_{j,0}).$$ 

Moreover, 
$\mathfrak{P}(y_{j,1},y_{j-1,0};u_{j,1})
= \mathfrak{P}(v_j^{-k_j}y_{j,1},y_{j-1,0};v_j^{-k_j}u_{j,1}v_j^{k_j})
= \mathfrak{P}(y_{j,0},y_{j-1,0};u_{j,0})
$.

Thus,
$\mathfrak{P}(y_{j,n_j},y_{j,0};u_{j,n_j})\subseteq \mathfrak{P}(y_{j,0},y_{j-1,0};u_{j,0})$. 
Since $l_j\in \mathfrak{P}(y_{j,1},y_{j,0};u_{j,1})\subseteq \mathfrak{P}(y_{j,0},y_{j-1,0};u_{j,0})$, and it does not belong to $\mathfrak{P}(y_{j,n_j},y_{j,0};u_{j,n_j})$, the inclusion is proper.

\end{proof}

By Lemma \ref{lmklcftf}, each set $\mathfrak{P}(y_{j+1,0},y_{j,0};u_{j+1,0})$ is non-empty. 
Moreover, $\mathfrak{P}(y_{0,0},y_{-1,0};u_{0,0})=\mathfrak{P}(y,x;g)$ is finite. The claim therefore gives an infinite strictly descending sequence of non-empty subsets of the finite set $\mathfrak{P}(y,x;g)$, which is impossible.

This contradiction shows that our initial assumption was false. Hence $\langle f,g\rangle$ contains an element of infinite order.

In the case of \css complexes, we use the additional assumption that $x$ and $y$ correspond to not essentially \css cells.  
Every application of Lemmas \ref{lmklcftf} and \ref{lem:doubling} is replaced by the corresponding application of Lemmas \ref{lmklcsv} and \ref{lem:doublingsys}, respectively. 
Moreover, in this case, the admissible exponent set $\mathfrak{P}(y,x;g)$ should be defined as the set of exponents $0<\ell<m$ such that $N(\mathfrak{X}_1(y,x))\cap g^\ell\mathfrak{X}_1(y,x)= \emptyset$.

\end{proof}

\section{Proofs of Theorems \ref{thm:tA}-\ref{thm:tB}}\label{sec:proofs}

\begin{proof}[Proof of Theorem \ref{thm:tB}]
If $G$ is trivial, there is nothing to prove. Assume that $G$ is nontrivial.
By Lemma \ref{lm:fixpt}, every element of the torsion group $G$ stabilizes a graphical cell of $X$. 
Choose a nontrivial element $f\in G$, and let $\Gamma_i$ be a graphical cell stabilized by $f$. 
In the \css setting, if the action is not essentially \css, we may choose $f$ so that $\Gamma_i$ is not essentially \css.

We claim that $G$ stabilizes $\Gamma_i$. Indeed, let $g\in G$. If $g$ does not stabilize $\Gamma_i$, then Lemma \ref{lem:inford} implies that $\langle f,g\rangle< G$ contains an element of infinite order, contradicting the assumption that $G$ is torsion.
Hence, every $g\in G$ stabilizes $\Gamma_i$.

Since $\Gamma_i$ is finite and the action of $G$ on the 1-skeleton of $X$ is free, the action of $G$ on the vertex set of $\Gamma_i$ is also free. 
Therefore, $G$ is finite.
\end{proof}

Theorems \ref{thm:tA} and \ref{thm:tD} are easy corollaries.

\begin{proof}[Proofs of Theorem \ref{thm:tA} and \ref{thm:tD}]
Let $\langle f:\Gamma \rightarrow \Theta \rangle$ be a graphical presentation of $G$, and let $\tilde{X}^{+}$ be the simplified universal cover of the graphical complex $X$.

It is clear that action of $G$ on $\tilde{X}^{+}$ is free on a $1$-skeleton. This property is inherited by any torsion subgroup $H$ of $G$. Therefore, by Theorem \ref{thm:tB}, $H$ is finite.

\end{proof}

\section{Automatic Continuity}\label{sec:ac}

In this section, we present an aforementioned application of Theorem~\ref{thm:tB} to automatic continuity.

For a group $G$ and a metric space $X$, a group action $\Phi:G\rightarrow Isom(X)$ is called \textbf{proper} if for each $x\in X$ there exist a real number $r>0$ such that the set $\{g\in G | (B_r(\Phi(g)(x))\cap B_r(x))\neq 0\}$ is finite.
The group action $\Phi$ is called \textbf{cocompact} if there exists a compact subset $K\subseteq X$ such that $\Phi(G)(K)=X$. We say that $\Phi$ is a \textbf{geometric} action if it is both proper and cocompact. We say that a group $G$ is \textbf{systolic} if it admits a geometric action on a systolic complex.

A group $G$ is called \textbf{artinian} if any descending chain of subgroups $G_1\supset G_2\supset \ldots$ becomes stationary, that is, $G_n=G_{n+1}=\ldots$ from some $n$ onwards.

Let $\mathcal{G}$ denote the class of all groups $G$ with the following three properties:
\begin{enumerate}
    \item $G$ does not include $\mathbb{Q}$ or the $p$-adic integers $\mathbb{Z}_p$ for any prime as a subgroup;
    \item $G$ does not include the Pr\"ufer $p$ group $\mathbb{Z}(p^{\infty})$ for any prime as a subgroup;
    \item torsion subgroups in $G$ are artinian.
\end{enumerate}
The following fact is known by Prop 11.1 of \cite{dudaall3}.

\begin{pr}
\label{p:automatic}
If $G$ is a systolic group whose torsion subgroups are artinian then $G$ is in the class $\mathcal{G}$.
\end{pr}

\begin{proof}[Proof of Theorem \ref{thm:tC}]
By \cite[Theorem B]{keppeler2021automatic}, the statement of Theorem \ref{thm:tC} holds for any group from the class $\mathcal{G}$, therefore, the question reduces to showing that $G$ belongs to this class. 

By Proposition \ref{p:automatic}, it suffices to show that $G$ is systolic and every torsion subgroup of $G$ is finite. The second part follows directly from Theorem \ref{thm:tB}. The first part follows from the fact that $G$ admits a geometric action on a graphical \css complex, inducing a geometric action on its corresponding Wise complex, which is systolic by Theorem \ref{t:ss}. 
\end{proof}

\subsection*{Acknowledgements} 
We thank Damian Osajda for helpful discussions.
Part of this work was carried out while K.D. was visiting the GGT research group in Copenhagen. K.D. gratefully acknowledges their hospitality.
Part of this work was carried out while H.G. was visiting the GTA research group in Bilbao. H.G. gratefully acknowledges their hospitality.
The work presented here was partially supported by the Carlsberg Foundation, grant CF23-1226. 
The work presented here was partially supported by the Basque Government, grant IT1483-22. 

\bibliographystyle{amsalpha}
\bibliography{ref}

\end{document}